\documentclass[11pt]{article}
\usepackage[utf8]{inputenc}
\usepackage[T1]{fontenc}
\usepackage[frenchb,english]{babel} 
\usepackage{textcomp}
\usepackage{amsmath,amssymb}
\usepackage{amsthm}
                  
\usepackage{lmodern}

\usepackage[a4paper,margin=2cm]{geometry}

\usepackage{graphicx}             
\usepackage{xcolor}               
\usepackage{microtype,stmaryrd}          

\usepackage{hyperref}
\hypersetup{pdfstartview=XYZ}

\usepackage{bbm}
\usepackage{mathrsfs}
\usepackage{subfig}

\def\build#1_#2^#3{\mathrel{
\mathop{\kern 0pt#1}\limits_{#2}^{#3}}}
\def\llbracket{[\hspace{-.10em} [ }
\def\rrbracket{ ] \hspace{-.10em}]}

\newtheorem{theorem}{Theorem}
\newtheorem{proposition}[theorem]{Proposition}

\newtheorem{lemma}[theorem]{Lemma}

\def\w{\mathrm{w}}

\def\t{\mathcal{T}}

\def\f{\mathcal{F}}
\def\r{\mathcal{R}}

\def\N{\mathbb{N}}
\def\M{\mathbb{M}}

\def\D{\mathbb{D}}
\def\P{\mathbb{P}}
\def\U{\mathbb{U}}
\def\E{\mathbb{E}}

\def\R{\mathbb{R}}
\def\z{\mathcal{Z}}

\def\n{\mathcal{N}}

\def\ee{\mathcal{E}}
\def\ve{{\varepsilon}}
\def\la{\longrightarrow}
\def\da{\downarrow}

\def\dd{\mathrm{d}}
\def\wh{\widehat}
\def\wt{\widetilde}
\def\tr{\mathrm{tr}}

\def\v{\mathcal{V}}

\def\xx{\mathbf{x}}
\def\kk{\mathcal{K}}

\def\be{\mathbf{e}}

\def\HH{\mathfrak{H}}

\author{Jean-Fran\c cois Le Gall and Armand Riera}
\title{Peeling the Brownian half-plane}
\date{\small Universit\'e Paris-Saclay and Sorbonne Universit\'e}

\begin{document}

\maketitle

\begin{abstract}
We establish a new spatial Markov property of the Brownian half-plane. According to this property,
if one removes a hull centered at a boundary point, the remaining space
equipped with an intrinsic metric is still a Brownian half-plane, which is
independent of the part that has been removed. This is an analog of the
well-known peeling procedure for random planar maps. We also
investigate several distributional properties of hulls centered at
a boundary point, and we provide a 
new construction of the Brownian half-plane giving information 
about distances from a half-boundary.
\end{abstract}

%\tableofcontents
\section{Introduction}

This work is concerned with the models of random geometry that arise as scaling limits of large 
graphs embedded in the sphere. We are especially interested in the Brownian half-plane, which appears as the scaling
limit of large planar quadrangulations with a boundary \cite{BMR}, under appropriate conditions on their boundary sizes, or of the 
infinite random lattice known as the Uniform Infinite Half-Planar Quadrangulation \cite{GM}. The Brownian half-plane 
shares the same local properties as the other well-known models called the Brownian sphere, the Brownian disk or
the Brownian plane, but it has the topology of the usual half-plane. In particular, one may define its boundary as the
set of all points that have no neighborhood homeomorphic to the open unit disk. One of the main results of
the present work is a new spatial Markov property that can be described informally as follows. Let $r>0$, and suppose that
one has explored all points of the Brownian half-plane that are at distance smaller than or equal to $r$ from a distinguished point of the
boundary, and also all points that have been disconnected from infinity in this exploration. Then the remaining 
unexplored region is still a Brownian half-plane, which furthermore is independent of the explored region. Different forms
of the spatial 
Markov property have been obtained in other models (see in particular \cite{spine,LGR3}), but this property takes a nicer form
for the Brownian half-plane, where no conditioning is needed to describe the law of the unexplored region. The spatial
Markov property of the Brownian half-plane is a kind of continuous version of the peeling process of
(finite or infinite) random planar maps, which involves exploring faces of the map one after another, in such a way that
the distribution of the unexplored region only depends on the size of its boundary. The peeling process of random
planar maps has found a large number of applications (see in particular \cite{AC,Cur}), and we hope that similar applications can
be developed in our continuous setting (see \cite{FS} for a first application of the spatial Markov property of the Brownian half-plane).

Let us give a more precise presentation of our main results. We write $(\mathfrak{H},D)$ for the 
random metric space that we call the Brownian half-space. This space comes with a volume measure denoted
by $V$ and a distinguished curve $(\Lambda(t))_{t\in\R}$ whose range is the
boundary $\partial\mathfrak{H}$. Therefore, we view $(\mathfrak{H},D,V,\Lambda)$ as a 
curve-decorated measure metric space, in the framework of \cite{GM}. The distinguished point on the boundary is
the point $\xx=\Lambda(0)$ (this point plays no special role and could be replaced by $\Lambda(t)$, for a given $t\in\R$, in what
follows). For $r>0$, we consider  the closed ball of radius $r$ centered at $\xx$, which
we denote by $B_r(\mathfrak{H})$.
The hull $B^\bullet_r(\mathfrak{H})$ is obtained by filling in the ``holes'' of $B_r(\mathfrak{H})$, or, more precisely,
$B^\bullet_r(\mathfrak{H})$ is the complement of the unbounded component of the open set $\mathfrak{H}\backslash B_r(\mathfrak{H})$.
Write $\mathfrak{H}_r$ for the closure of $\mathfrak{H}\backslash B^\bullet_r(\mathfrak{H})$. According to Theorem \ref{peeling-HP}, $\mathfrak{H}_r$
equipped with an intrinsic distance $D_r$, with the restriction $V_r$ of the volume measure $V$,
and  with an appropriately defined boundary curve $\Lambda^r$,  is again a Brownian half-plane. Furthermore, $\mathfrak{H}_r$
is independent of the hull $B^\bullet_r(\mathfrak{H})$ also viewed as a random
curve-decorated measure metric space (Theorem \ref{indep-peeling}). In the same way as it was done for
the peeling of random planar maps, one can iterate the procedure and remove from $\mathfrak{H}_r,$ a hull
centered at a boundary point $\mathbf{\xx'}$ of $\partial\mathfrak{H}_r$ (which can be chosen as a measurable function of $B^\bullet_r(\HH)$)
to get another Brownian half-plane, and so on. See Fig.~1 for a schematic illustration.

\smallskip
 \begin{figure}[!h]
\centering
   \includegraphics[width=12cm]{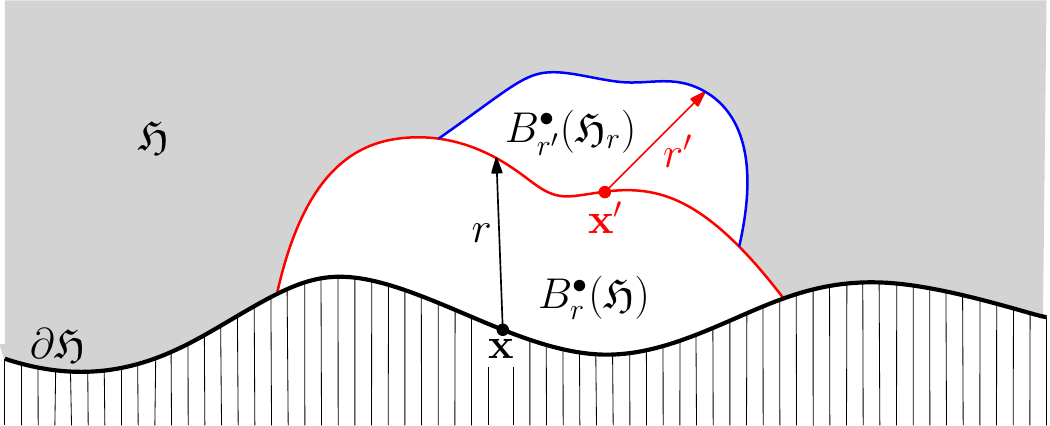}
    \caption{Iterating the peeling procedure. The grey part is still a Brownian half-plane. The hull $B^\bullet_{r'}(\mathfrak{H}_r)$ is here centered
    at a point $\xx'$ which can be chosen as a function of $B^\bullet_r(\mathfrak{H})$.}
  \end{figure}

Motivated by the spatial Markov property described above, we compute several exact distributions related to the hull $B^\bullet_r(\HH)$.
In particular, the perimeter $Z_r$ of the hull $B^\bullet_r(\HH)$ (measuring the ``size'' of the topological boundary of the hull) is
exponentially distributed with mean $2r^2/3$. More precisely, Proposition \ref{formulas!} 
gives the joint distribution of $Z_r$, of the volume $V(B^\bullet_r(\HH))$,
and of the sizes of the subsets of $\partial\HH$ that have been ``swallowed'' by the hull on both sides of the
distinguished point $\xx$. Furthermore, we also study the perimeter process $(Z_r)_{r>0}$ and characterize its distribution 
as that of  a self-similar Markov process, which is associated via Lamperti's representation with a spectrally negative L\'evy process
whose Laplace exponent has a simple explicit form (Proposition \ref{peri-pro-ss}). 

The proof of our spatial Markov property (Theorems \ref{peeling-HP} and \ref{indep-peeling}) relies on related 
results obtained for the Brownian disk in \cite{LGR3}, and on a coupling between  the Brownian half-plane and
the Brownian disk with a large boundary size. As another application of this coupling, we derive a new representation 
of the Brownian half-plane $(\mathfrak{H},D,V,\Lambda)$ in Theorem \ref{new-const-half-plane}. Like previous constructions, this representation
involves random $\R$-trees equipped with (nonnegative) Brownian labels, However, in contrast with 
preceding work, labels now correspond to distances from the  ``negative half-boundary'' $\partial_1\HH:=\{\Lambda(t):t\leq 0\}$.
As a consequence, the process $(D(\Lambda(t),\partial_1\HH))_{t\geq 0}$ is distributed as a three-dimensional 
Bessel process. This should be compared with the identification of the processes $(D(\Lambda(t),\Lambda(0))_{t\geq 0}$ 
and $(D(\Lambda(-t),\Lambda(0))_{t\geq 0}$ as two independent five-dimensional Bessel processes, which follows
from the Caraceni-Curien construction of the Brownian half-plane \cite{CC}.  As a side remark, there are several constructions of the 
Brownian half-plane in terms of labeled trees, where labels may correspond either to distances from a distinguished 
point of the boundary \cite{CC}, or to ``distances from infinity'' \cite{BMR,GM}, or to distances from the boundary \cite{spine}. 
These constructions provide different pieces of information on the Brownian half-plane, but it is somewhat intriguing that
there is no simple direct way to prove that they all give rise to the same object (for this, one typically needs to
come back to the discrete approximations).

The spatial Markov property of Theorems \ref{peeling-HP} and \ref{indep-peeling} suggests several questions. In particular, one
may ask about describing a general class of simply connected open subsets $U$ of $\HH$ whose intersection with the boundary is an ``interval'' $\{\Lambda(t):\alpha<t<\beta\}$,
such that $\HH\backslash U$ equipped with an intrinsic metric is again a Brownian half-plane, which furthermore is independent of $U$
in an appropriate sense (there is an obvious analogy with the strong Markov property of Brownian motion). Iterating the peeling procedure described above gives examples of such subsets $U$ (see also Section \ref{sec:segment}). Another appealing but still vague problem
would be to characterize the Brownian half-plane by such a general form of the spatial Markov property, together with other properties 
to be specified. 

The paper is organized as follows. Section \ref{sec:preli} gives some preliminaries concerning curve-decorated metric spaces
and the Brownian snake excursion measures $\N_x$, which are our basic tools to construct Brownian trees 
equipped with Brownian labels. In Section \ref{Spa-Mar}, we recall the Caraceni-Curien construction of the 
Brownian half-plane, and we prove the spatial Markov property (Theorems \ref{peeling-HP} and \ref{indep-peeling}). 
Section \ref{sec:explicit} is devoted to the calculation of explicit distributions related to the hulls $B^\bullet_r(\HH)$.
In the same direction, two different characterizations of the perimeter process $(Z_r)_{r>0}$
are discussed in Section \ref{sec:peripro}. Finally, Section \ref{sec:new} presents our new
construction of the Brownian half-plane, and Section \ref{sec:segment} discusses an analog
of Theorem \ref{peeling-HP} for hulls centered on a segment of the boundary. 
\\
\\
\noindent{\bf Main notation.}
{\setlength{\leftmargini}{1em}
\begin{itemize}
  \item $\M^{\mathrm{GHPU}}$ space of curve-decorated compact measure metric spaces (Sect. 2.1)
  \vspace{-2mm}\item $\M^{\mathrm{GHPU}}_\infty$ space of curve-decorated boundedly compact measure length spaces (Sect. 2.1)
 \vspace{-2mm}\item  $\mathfrak{B}_r(\mathfrak{X})$ ball in $\mathfrak{X}$ viewed as a curve-decorated measure metric space (Sect. 2.1)
 \vspace{-2mm}\item $\mathfrak{W}$ set of all finite paths,\; $\mathfrak{W}_x$ set of all finite paths started at $x$ (Sect. 2.2)
\vspace{-2mm}\item $\zeta_{(\w)}$ lifetime of $\w\in\mathfrak{W}$ (Sect. 2.2)
\vspace{-2mm}\item $\wh\w=\w(\zeta_{(\w)})$ endpoint of $\w\in\mathfrak{W}$ (Sect. 2.2)
\vspace{-2mm}\item $\mathcal{S}$ set of all snake trajectories  (Sect. 2.2) 
 \vspace{-2mm}\item $(W_s)_{s\ge0}$ canonical process on $\mathcal{S}$ (Sect. 2.2)  
 \vspace{-2mm}\item $(\zeta_s)_{s\geq 0}$ lifetime process on $\mathcal{S}$ (Sect. 2.2)
 \vspace{-2mm}\item $\t_{(\omega)}$  {genealogical tree} of the snake trajectory $\omega\in \mathcal{S}$ (Sect. 2.2)
 \vspace{-2mm}\item $p_{(\omega)}$  canonical projection from $[0,\sigma(\omega)]$ onto $\t_{(\omega)}$ 
 \vspace{-2mm}\item $\ell_u(\omega):=\widehat{W}_s(\omega)$, for $s$ such that $p_{(\omega)}(s)=u$,  the label of $u\in \t_{(\omega)}$ (Sect. 2.2)  
\vspace{-2mm}\item $W_*(\omega)= \min\{\ell_u(\omega):~u\in\t_{(\omega)}\}$ the minimal label on $\t_{(\omega)}$ (Sect. 2.2)
\vspace{-2mm}\item $\tr_y(\omega)$ truncation of the snake trajectory at $y$ (Sect. 2.2)
\vspace{-2mm}\item $\N_x$ Brownian snake excursion measure from $x$ (Sect. 2.3)  
\vspace{-2mm}\item $(L^y_t)_{t\geq 0}$ exit local time at $y$ (Sect. 2.3)
\vspace{-2mm}\item $\mathcal{Z}_y$ exit measure at $y$ (Sect. 2.3)
\vspace{-2mm}\item $(\mathcal{R}_t)_{t\in\R}$ two-sided five-dimensional Bessel process started from $0$ (Sect. 3.1)
\vspace{-2mm}\item $\mathfrak{T}$ non-compact labeled $\mathbb{R}$-tree used to construct the Brownian half-plane (Sect. 3.1)
\vspace{-2mm}\item $(\mathcal{E}_t)_{t\in\R}$ clockwise exploration of $\mathfrak{T}$ (Sect. 3.1)
\vspace{-2mm}\item $(\mathfrak{H},D,V,\Lambda)$ curve-decorated Brownian half-plane associated with $\mathfrak{T}$ (Sect. 3.1)
\vspace{-2mm}\item $\Pi$ canonical projection from $\mathfrak{T}$ onto $\mathfrak{H}$ (Sect. 3.1)
\vspace{-2mm}\item $B_r(\mathfrak{H})$ closed ball of radius $r$ centered at $\Lambda(0)$ in $\mathfrak{H}$ (Sect. 3.1)
\vspace{-2mm}\item $B_r^\bullet(\mathfrak{H})$ hull of radius $r$ centered at $\Lambda(0)$ in $\mathfrak{H}$ (Sect. 3.1)
\vspace{-2mm}\item $\beta_r$ last passage time at $r$ of $(\sqrt{3}\,\mathcal{R}_{-t})_{t\geq 0}$ (Sect. 3.1) 
\vspace{-2mm}\item $\gamma_r$ last passage time at $r$ of $(\sqrt{3}\,\mathcal{R}_{t})_{t\geq 0}$ (Sect. 3.1) 
\vspace{-2mm}\item $Z_r$ perimeter of the hull $B_r^\bullet(\mathfrak{H})$ (Sect. 3.1)
\vspace{-2mm}\item $\mathfrak{H}_r$ closure of the complement of $B_r^\bullet(\mathfrak{H})$  (Sect. 3.1)
\vspace{-2mm}\item $\Lambda^{r}$ boundary curve of  $\mathfrak{H}_r$ (Sect. 3.1)
\vspace{-2mm}\item $\mathcal{V}_r$ volume of the hull $B_r^\bullet(\mathfrak{H})$ (Sect. 4.1)
\vspace{-2mm}\item  $(\be_t)_{0\leq t\leq 1}$ normalized Brownian excursion (Sect. 6.1)
\vspace{-2mm}\item $\mathfrak{T}^\star$ labeled tree used to construct the Brownian disk (Sect. 6.1)
\vspace{-2mm}\item $(\U,D_\star, \mathbf{V}_\star,\Pi_\star(0))$  free Brownian disk constructed from $\mathfrak{T}^\star$ (Sect. 6.1)
\vspace{-2mm}\item $B^\bullet_r(\mathfrak{H},[0,s])$ hull centered on the boundary segment $\Lambda([0,s])$ of $\mathfrak{H}$ (Sect. 7)
\end{itemize}}

\section{Preliminaries}

\label{sec:preli}

\subsection{Curve-decorated spaces}
\label{curve-deco}

We are interested in  (random) metric spaces equipped with additional structures and follow closely \cite{GM}. 
If $(E,d)$ is a compact metric space, we let $C_0(\R,E)$ be the space
of all continuous functions $\gamma:\R\la E$ such that, for every $\ve>0$,
there exists $T>0$ such that $d(\gamma(t),\gamma(T))<\ve $
and $d(\gamma(-t),\gamma(-T))<\ve $ for every $t\geq T$. By convention, if
$\gamma:[a,b]\la E$ is only (continuous and) defined on an interval $[a,b]$,
we view it as an element of $C_0(\R,E)$ by extending it so that it is constant on $(-\infty,a]$
and on $[b,\infty)$. 
Following \cite{GM}, we say that a curve-decorated (compact) measure 
metric space is a compact metric space $(E,d)$ equipped with a finite Borel measure $\mu$
(sometimes called the volume measure) and
with a curve $\gamma\in C_0(\R,E)$ --- we then often view $\gamma(0)$ as
the distinguished point of $E$. We write $\M^{\mathrm{GHPU}}$ for the set of all
isometry classes of curve-decorated compact measure metric spaces (here $(E,d,\mu,\gamma)$
and $(E',d',\mu',\gamma')$ are isometry equivalent if there exists an isometry $\Phi$ from 
$E$ onto $E'$ such that $\Phi_*\mu=\mu'$ and $\gamma'=\Phi\circ\gamma$).
One can equip $\M^{\mathrm{GHPU}}$ with the so-called
Gromov-Hausdorff-Prokhorov-uniform distance $\dd_{\mathrm{GHPU}}$, which is defined by
\begin{align*}
&\dd_{\mathrm{GHPU}}\big((E_1,d_1,\mu_1,\gamma_1),(E_2,d_2,\mu_2,\gamma_2)\big)\\
&:=\inf\Big\{d^E_{\mathrm{H}}\big(\Phi_1(E_1),\Phi_2(E_2)\big) \vee d^E_\mathrm{P}\big((\Phi_{1})_*\mu_1 ,(\Phi_2)_*\mu_2\big)
\vee\sup_{t\in\R}d\big(\Phi_1\circ\gamma_1(t),\Phi_2\circ\gamma_2(t)\big)\Big\},
\end{align*}
where the infimum is over all isometric embeddings $\Phi_1:E_1\la E$ and $\Phi_2:E_2\la E$ of 
$E_1$ and $E_2$ into the same compact metric space $(E,d)$, $d^E_\mathrm{H}$ is the usual Hausdorff distance 
between compact subsets of $E$, and $d^E_{\mathrm{P}}$
denotes the Prokhorov metric on the space of all finite Borel measures on $E$. By \cite[Proposition 1.3]{GM}, 
$\dd_{\mathrm{GHPU}}$ is a complete
separable metric on $\M^{\mathrm{GHPU}}$.

We will also be interested in non-compact metric spaces. We again follow closely \cite{GM} and restrict ourselves to
length spaces for technical reasons. Recall that a metric space is called boundedly compact if
all closed balls are compact. We then let $\M^{\mathrm{GHPU}}_\infty$ denote the set 
of all (isometry classes of) $4$-tuples $\mathfrak{X}=(X,d,\mu,\gamma)$, where $(X,d)$ is a boundedly compact
length space, $\mu$ is a Borel measure on $X$ that assigns finite mass to
every compact subset of $X$, and $\gamma:\R\la X$
is a continuous curve in $X$. As previously, we identify $(X,d,\mu,\gamma)$ and $(X',d',\mu',\gamma')$
if there is an isometry $\Phi:X\la X'$ such that $\Phi_*\mu=\mu'$ and $\gamma'=\Phi\circ\gamma$.
We can then define a local version of $\dd_{\mathrm{GHPU}}$ as follows. If $\mathfrak{X}=(X,d,\mu,\gamma)
\in \M^{\mathrm{GHPU}}_\infty$, we first need to define the ball of radius $r>0$ of $\mathfrak{X}$
as an element of $\M^{\mathrm{GHPU}}$. To this end, for every $r>0$, we define
$$\underline{\tau}^\gamma_r :=(-r)\vee\sup\{t<0:d(\gamma(0),\gamma(t))=r\}\,,\quad 
\overline{\tau}^\gamma_r :=r\wedge\inf\{t>0:d(\gamma(0),\gamma(t))=r\},$$
with the usual conventions $\sup\varnothing=-\infty$ and $\inf\varnothing=+\infty$. We then define
$\mathfrak{B}_r\gamma\in C_0(\R,X)$ by setting $\mathfrak{B}_r\gamma(t)=\gamma((t\wedge \overline{\tau}^\gamma_r)\vee \underline{\tau}^\gamma_r )$
for every $t\geq 0$. Finally we define the ``ball'' $\mathfrak{B}_r(\mathfrak{X})$ as the curve-decorated (compact)
measure metric space
$$\mathfrak{B}_r(\mathfrak{X}):=(B_r(X),d_{|B_r(X)}, \mu_{|B_r(X)},\mathfrak{B}_r\gamma),$$
where $B_r(X)$ denotes the closed ball of radius $r$ centered at $\gamma(0)$ in $X$. The local Gromov-Hausdorff-Prokhorov-uniform distance
on $\M^{\mathrm{GHPU}}_\infty$ is then defined by
$$\dd_{\mathrm{GHPU}}^\infty(\mathfrak{X},\mathfrak{X}'):=\int_0^\infty e^{-r}\,\Big(1\wedge \dd_{\mathrm{GHPU}}\big(\mathfrak{B}_r(\mathfrak{X}),\mathfrak{B}_r(\mathfrak{X}')\big)\Big)\,\dd r.$$
According to \cite[Proposition 1.7]{GM}, $\dd_{\mathrm{GHPU}}^\infty$ is a complete separable metric on $\M^{\mathrm{GHPU}}_\infty$.

We will also need to deal with pointed measure metric spaces, which simply amounts to considering the special case 
where the decorating curve $\gamma$ is constant. Both in the compact and in the non-compact case, we can 
view the class of pointed spaces as a subclass of $\M^{\mathrm{GHPU}}$, resp. of $\M^{\mathrm{GHPU}}_\infty$, which we
equip with the restriction of the distance $\dd_{\mathrm{GHPU}}$, resp. of $\dd_{\mathrm{GHPU}}^\infty$. With any
curve-decorated space $(X,d,\mu,\gamma)$ we can associate the pointed space $(X,d,\mu,\gamma(0))$, and this
mapping is trivially continuous.

\subsection{Snake trajectories}
\label{sna-tra}

We now briefly present the formalism 
of snake trajectories that we will use to define our models of random geometry
(we refer to \cite{ALG} for more details).
A (one-dimensional) finite path $\w$ is a continuous mapping $\w:[0,\zeta]\la\R$, where  $\zeta=\zeta_{(\w)}\geq 0$ is called the lifetime of $\w$. We let 
$\mathfrak{W}$ denote the space of all finite paths, which is equipped with the
distance
$d_\mathfrak{W}(\w,\w'):=|\zeta_{(\w)}-\zeta_{(\w')}|+\sup_{t\geq 0}|\w(t\wedge
\zeta_{(\w)})-\w'(t\wedge\zeta_{(\w')})|$.
We denote the tip or endpoint of a path $\w\in \mathfrak{W}$ by $\wh \w=\w(\zeta_{(\w)})$, and
for $x\in\R$, we
set $\mathfrak{W}_x:=\{\w\in\mathfrak{W}:\w(0)=x\}$. The trivial element of $\mathfrak{W}_x$ 
with zero lifetime is identified with the point $x$ of $\R$. 

\smallskip
\noindent {\bf Definition}. Let $x\in\R$. 
A snake trajectory with initial point $x$ is a continuous mapping $s\mapsto \omega_s$
from $\R_+$ into $\mathfrak{W}_x$ 
that satisfies the following two properties:
\begin{enumerate}
\item[\rm(i)] We have $\omega_0=x$ and the number $\sigma(\omega):=\sup\{s\geq 0: \omega_s\not =x\}$,
called the duration of the snake trajectory $\omega$,
is finite (by convention $\sigma(\omega)=0$ if $\omega_s=x$ for every $s\geq 0$). 
\item[\rm(ii)] {\rm (Snake property)} For every $0\leq s\leq s'$, we have
$\omega_s(t)=\omega_{s'}(t)$ for every $t\in[0,\displaystyle{\min_{s\leq r\leq s'}} \zeta_{(\omega_r)}]$.
\end{enumerate} 

\smallskip
We write $\mathcal{S}_x$ for the set of all snake trajectories with initial point $x$
and $\mathcal{S}=\bigcup_{x\in\R}\mathcal{S}_x$ for the set of all snake trajectories. If $\omega\in \mathcal{S}$, we often write $W_s(\omega):=\omega_s$ and $\zeta_s(\omega):=\zeta_{(\omega_s)}$
for every $s\geq 0$.We equip $\mathcal{S}$ with the distance
$d_{\mathcal{S}}(\omega,\omega'):= |\sigma(\omega)-\sigma(\omega')|+ \sup_{s\geq 0} \,d_\mathfrak{W}(W_s(\omega),W_{s}(\omega'))$.
We notice that a snake trajectory $\omega$ is determined 
by the knowledge of the lifetime function $s\mapsto \zeta_s(\omega)$ and of the tip function $s\mapsto \wh W_s(\omega)$: See \cite[Proposition 8]{ALG}.

Let $\omega\in \mathcal{S}$ be a snake trajectory and $\sigma=\sigma(\omega)$. The lifetime function $s\mapsto \zeta_s(\omega)$ codes a
compact $\R$-tree, which will be denoted 
by $\t_{(\omega)}$ and called the {\it genealogical tree} of the snake trajectory. This $\R$-tree is the quotient space $\t_{(\omega)}:=[0,\sigma]/\!\sim$ 
of the interval $[0,\sigma]$
for the equivalence relation defined by setting
$s\sim s'$ if and only if $\zeta_s(\omega)=\zeta_{s'}(\omega)= \min_{s\wedge s'\leq r\leq s\vee s'} \zeta_r(\omega)$. We then equip
$\t_{(\omega)}$ with the distance induced by
$$d_{(\omega)}(s,s'):= \zeta_s(\omega)+\zeta_{s'}(\omega)-2 \min_{s\wedge s'\leq r\leq s\vee s'} \zeta_r(\omega),$$
 and we stress that $d_{(\omega)}(s,s')=0$ if and only if $s\sim s'$.
We write $p_{(\omega)}:[0,\sigma]\la \t_{(\omega)}$
for the canonical projection, and note that the mapping $[0,\sigma]\ni t\mapsto p_{(\omega)}(t)$ can be viewed as a
cyclic exploration of $\t_{(\omega)}$. By convention, $\t_{(\omega)}$ is rooted at the point
$\rho_{(\omega)}:=p_{(\omega)}(0)$, and the volume measure on $\t_{(\omega)}$ is defined as the pushforward of
Lebesgue measure on $[0,\sigma]$ under $p_{(\omega)}$. If $u,v\in\t_{(\omega)}$, we write $\llbracket u,v\rrbracket$ for
the geodesic segment between $u$ and $v$ in $\t_{(\omega)}$.
 The segment $\llbracket \rho_{(\omega)},u\rrbracket$
 is called the ancestral line of $u$, and we say that $u$ is a descendant of $v$
 if $v\in \llbracket \rho_{(\omega)},u\rrbracket$.

By property (ii) in the definition of  a snake trajectory, the condition $p_{(\omega)}(s)=p_{(\omega)}(s')$ implies that 
$W_s(\omega)=W_{s'}(\omega)$. So the mapping $s\mapsto W_s(\omega)$ can be viewed as defined on the quotient space $\t_{(\omega)}$.
For $u\in\t_{(\omega)}$, we set $\ell_u(\omega):=\wh W_s(\omega)$ for any $s\in[0,\sigma]$ such that $u=p_{(\omega)}(s)$, and we interpret $\ell_u(\omega)$ as a ``label'' assigned to the point $u$ of $\t_{(\omega)}$. 
We note  that the mapping $u\mapsto \ell_u(\omega)$ is continuous on $\t_{(\omega)}$, and we set $W_*(\omega):=\min\{\ell_u(\omega):u\in\t_{(\omega)}\}$.
We also observe that, for every $s\in[0,\sigma]$, the path $W_s(\omega)$ 
records the labels along the ancestral line of $p_{(\omega)}(s)$.

We next introduce the important truncation operation on snake trajectories. 
Let $x,y\in \R$ with $y<x$. For every $\w\in\mathfrak{W}_x$, set
$\tau_y(\w):=\inf\{t\in[0,\zeta_{(\w)}]: \w(t)=y\}\leq +\infty$.
Then, if 
$\omega\in \mathcal{S}_x$, we set, for every $s\geq 0$,
$$\eta_s(\omega):=\inf\Big\{t\geq 0:\int_0^t \mathrm{d}r\,\mathbf{1}_{\{\zeta_{(\omega_r)}\leq\tau_y(\omega_r)\}}>s\Big\}.$$
Note that the condition $\zeta_{(\omega_r)}\leq\tau_y(\omega_r)$ holds if and only if $\tau_y(\omega_r)=\infty$ or $\tau_y(\omega_r)=\zeta_{(\omega_r)}$.
Then, setting $\omega'_s=\omega_{\eta_s(\omega)}$ for every $s\geq 0$ defines an element $\omega'$ of $\mathcal{S}_x$,
which will be denoted by  $\tr_y(\omega)$ and called the truncation of $\omega$ at $y$
(see \cite[Proposition 10]{ALG}). The effect of the time 
change $\eta_s(\omega)$ is to ``eliminate'' those paths $\omega_s$ that hit $y$ and then survive for a positive
amount of time. 
Informally, $\t_{(\tr_y(\omega))}$ is obtained from 
$\t_{(\omega)}$ by pruning branches at the level where labels first take the value $y$. 

We finally introduce the excursions of a snake trajectory away from a given level. Let $\omega\in \mathcal{S}_x$
and  $y<x$. Let $(\alpha_j,\beta_j)$, $j\in J$, be the connected components of the open set
$\{s\in[0,\sigma]:\tau_y(\omega_s)<\zeta_{(\omega_s)}\}$,
and notice that, for every $j\in J$, we have $\omega_{\alpha_j}=\omega_{\beta_j}$, and $\zeta_s>\zeta_{(\omega_{\alpha_j})}$ for every $s\in(\alpha_j,\beta_j)$.
For every $j\in J$, we define a snake trajectory $\omega^j\in\mathcal{S}_y$ by setting
$$\omega^j_{s}(t):=\omega_{(\alpha_j+s)\wedge\beta_j}(\zeta_{(\omega_{\alpha_j})}+t)\;,\hbox{ for }0\leq t\leq \zeta_{(\omega^j_s)}
:=\zeta_{(\omega_{(\alpha_j+s)\wedge\beta_j})}-\zeta_{(\omega_{\alpha_j})}\hbox{ and } s\geq 0.$$
We say that $\omega^j$, $j\in J$, are the excursions of $\omega$ away from $y$. We note that, for every $j\in J$,
the tree $\t_{(\omega^j)}$ is canonically identified to a subtree of $\t_{(\omega)}$ consisting of
descendants of $p_{(\omega)}(\alpha_j)=p_{(\omega)}(\beta_j)$. 

\subsection{The Brownian snake excursion 
measure on snake trajectories}
\label{sna-mea}

Let $x\in\R$. The Brownian snake excursion 
measure $\N_x$ is the $\sigma$-finite measure on $\mathcal{S}_x$ that satisfies the following two properties: Under $\N_x$,
\begin{enumerate}
\item[(i)] the distribution of the lifetime function $(\zeta_s)_{s\geq 0}$ is the It\^o 
measure of positive excursions of linear Brownian motion, normalized so that, for every $\ve>0$,
$$\N_x\Big(\sup_{s\geq 0} \zeta_s >\ve\Big)=\frac{1}{2\ve};$$
\item[(ii)] conditionally on $(\zeta_s)_{s\geq 0}$, the tip function $(\wh W_s)_{s\geq 0}$ is
a Gaussian process with mean $x$ and covariance function 
$$K(s,s'):= \min_{s\wedge s'\leq r\leq s\vee s'} \zeta_r.$$
\end{enumerate}
The measure $\N_x$ is the excursion measure away from $x$ for the 
Markov process in $\mathfrak{W}_x$ called the Brownian snake.
We refer to 
\cite{Zurich} for a detailed study of the Brownian snake. 
For every $y<x$, we have
\begin{equation}
\label{hittingpro}
\N_x(W_*\leq y)={\displaystyle \frac{3}{2(x-y)^2}},
\end{equation}
where we recall the notation $W_*(\omega)$ for the minimal label on $\t_{(\omega)}$. See e.g. \cite[Lemma 2.1]{LGW}.

\paragraph{Exit measures.} Let $x,y\in\R$, with $y<x$. One shows \cite[Proposition 34]{Disks} that the limit
\begin{equation}
\label{formu-exit}
L^y_t(\omega):=\lim_{\ve \da 0} \frac{1}{\ve^2} \int_0^t \dd s\,\mathbf{1}_{\{\tau_y(W_s(\omega))=\infty,\, \wh W_s(\omega)<y+\ve\}}
\end{equation}
exists uniformly in $t\geq 0$, $\N_x(\dd\omega)$ a.e., and defines a continuous nondecreasing function, which is 
obviously constant on $[\sigma,\infty)$. 
The process $(L^y_t)_{t\geq 0}$ is called the exit local time at $y$, and the exit measure 
$\z_y$ is defined by $\z_y:=L^y_\infty=L^y_\sigma$. Informally, $\z_y$ measures ``how many'' paths 
$W_s$  hit $y$ and are stopped at that hitting time. Then, $\N_x$ a.e., the topological support of the measure 
$\dd L^y_t$ is exactly the set $\{s\in[0,\sigma]:\tau_y(W_s)=\zeta_s\}$, and, in particular, $\z_y>0$ if and only if one of the paths $W_s$ hits $y$. The definition of $\z_y$
is a special case of the theory of exit measures (see \cite[Chapter V]{Zurich} for this general theory). 

The exit measure $\z_y$ is a function of the truncated snake trajectory $\tr_y(\omega)$. Indeed, a time change 
argument shows that the same formula \eqref{formu-exit} applied to $\tr_y(\omega)$ instead of $\omega$ yields
a continuous limit $t\mapsto L^y_t(\tr_y(\omega))$ which is equal to $L^y_{\eta_t}(\omega)$, where $(\eta_s(\omega))_{s\geq 0}$ 
is the time change used to define 
$\tr_y(\omega)$ at the end of Section~\ref{sna-tra}. In particular, $L^y_\infty(\tr_y(\omega))=L^y_\infty(\omega)=\z_y(\omega)$. 

\paragraph{The special Markov property.} Recall the notation introduced in  Section \ref{sna-tra}: for $y<x$ and $\omega\in\mathcal{S}_x$, we write  $\omega^j$, $j\in J$, for  the excursions of $\omega$ 
away from $y$, and $(\alpha_j,\beta_j)$, $j\in J$, for the corresponding subintervals of $[0,\sigma]$. The special Markov property states that, under $\N_x$, conditionally on the truncation $\tr_y(\omega)$,  the point measure
\begin{equation}\label{measure:special:mark}
\sum \limits_{j\in J} \delta_{(L^y_{\alpha_j},\omega^{j})}(\dd t\, \dd\omega')
\end{equation}
is Poisson with intensity $\mathbf{1}_{[0,\z_y]}(t)\,\dd t \,\mathbb{N}_y(\dd \omega')$. We refer to the Appendix of \cite{subor} for a proof. In the following,
we will use a minor extension of this result. Suppose that $x>y>0$, then the special Markov property holds in exactly the
same form under $\N_x(\cdot\cap\{W_*>0\})$, provided that the intensity measure is replaced by $\mathbf{1}_{[0,\z_y]}(t)\,\dd t \,\mathbb{N}_y(\dd \omega'\cap\{W_*(\omega')>0\})$.
We leave the easy proof to the reader.

\section{Spatial Markov property in the Brownian half-plane}
\label{Spa-Mar}

\subsection{The Caraceni-Curien construction of the Brownian half-plane}
\label{car-cur}

In this section, we present the Caraceni-Curien construction of the Brownian half-plane, which we will
use to define the boundary curve of hulls. This construction was given in \cite{CC},
and it was shown in \cite{Repre}  that it is equivalent to the other construction proposed in \cite{GM,BMR} 
(still a different equivalent construction appears in \cite{spine}). 

Consider a process $\mathcal{R}=(\mathcal{R}_t)_{t\in\R}$ such that $(\mathcal{R}_t)_{t\geq 0}$ and $(\mathcal{R}_{-t})_{t\geq 0}$ are two
independent five-dimensional Bessel processes started from $0$. Conditionally
on the process $\r$, let $\mathcal{N}(\dd t\,\dd\omega)$ be a Poisson point measure
on $\R\times \mathcal{S}$ with intensity
$$2\,\mathbf{1}_{\{W_*(\omega)>0\}}\,\dd t\,\N_{\sqrt{3}\mathcal{R}_t}(\dd \omega).$$
We write
$$\mathcal{N}(\dd t\,\dd\omega)=\sum_{i\in I} \delta_{(t_i,\omega^i)}(\dd t\,\dd\omega).$$
We then let 
$\mathfrak{T}$ be the metric space which is obtained from the union
$$\R\cup \Bigg(\bigcup_{i\in I} \t_{(\omega^i)}\Bigg)$$
by identifying the root $\rho_{(\omega^i)}$ of $\t_{(\omega^i)}$ 
with the point $t_i$ of $\R$. The metric $\dd_{\mathfrak{T}}$ on $\mathfrak{T}$ is defined in the obvious manner, so that
$\mathfrak{T}$ equipped with $\dd_{\mathfrak{T}}$ is a (non-compact) $\R$-tree,
the restriction of $\dd_{\mathfrak{T}}$ to each tree $\t_{(\omega^i)}$ is 
the metric $d_{(\omega^i)}$, as defined in Section \ref{sna-tra}, and $\dd_{\mathfrak{T}}(u,v)=|v-u|$ if $u,v\in\R$. The volume measure on $\mathfrak{T}$
is the sum of the volume measures on $\t_{(\omega^i)}$, $i\in I$. 

We can define a clockwise exploration $(\ee_t)_{t\in\R}$
of $\mathfrak{T}$, informally by concatenating the functions $p_{(\omega^i)}$, $i\in I$,
in the order prescribed by the $t_i$'s, in such a way that $\ee_0= 0$, and
$\{\ee_t:~t\geq 0\}=\R_+\cup(\bigcup_{i\in I,t_i>0}\t_{(\omega^i)})$ (see \cite[Section 4.1]{LGR3} for a precise definition
in a slightly different setting). The exploration process $\ee$ allows us
to define ``intervals'' on $\mathfrak{T}$. For $u,v\in \R$, 
if $v<u$, we set $[u,v]_\infty:=[u,\infty)\cup(-\infty, v]$, and if $u\leq v$, we let $[u,v]_\infty:=[u,v]$ be the usual interval. Then, for any $a,b\in \mathfrak{T}$, there
is a smallest ``interval'' $[u,v]_\infty$ such that $\ee_u=a$ and $\ee_v=b$, and we set
$[a,b]_\infty:=\{\ee_t:t\in [u,v]_\infty\}$. 

Finally, we define labels on $\mathfrak{T}$. If $a\in \R$, we take $\ell_a:=\sqrt{3}\,\mathcal{R}_a$, and,
if $a\in \t_{(\omega^i)}$, we let $\ell_a$ be the label of $a$ in $\t_{(\omega^i)}$
(when $a$ is the root of $\t_{(\omega^i)}$ the two definitions are consistent). A straightforward extension of
Lemma 3.3 in \cite{hull}, using formula \eqref{hittingpro}, shows that labels on $\mathfrak{T}$ are transient, meaning that, for every $r\geq 0$, 
the set $\{a\in\mathfrak{T}:\ell_a\leq r\}$ is compact.

We  set, for every $a,b\in\mathfrak{T}$,
\begin{equation}
\label{D1}
D^\circ(a,b):=\ell_a+\ell_b - 2\max\Big(\min_{c\in[a,b]_\infty}\ell_c,\min_{c\in[b,a]_\infty}\ell_c\Big)
\end{equation}
and then
\begin{equation}
\label{D2}
D(a,b):=\inf_{a_0=a,a_1,\ldots,a_{p-1},a_p=b}\sum_{k=1}^p D^\circ(a_{k-1},a_k)
\end{equation}
where the infimum is over all choices of the integer $p\geq 1$
and of the points $a_1,\ldots,a_{p-1}$ in $\mathfrak{T}$. Obviously, $D(a,b)\leq D^\circ(a,b)$, and
one can prove that $D(a,b)=0$ if and only if $D^\circ(a,b)=0$. 

Then, $D$ is a pseudo-metric on $\mathfrak{T}$, and we set $\mathfrak{H}:=\mathfrak{T}/\{D=0\}$,
where (here and later) we use the notation $\mathfrak{T}/\{D=0\}$ to denote the quotient 
space of $\mathfrak{T}$ for the equivalence relation defined by saying that
$a\approx b$ if and only if $D(a,b)=0$.
We equip $\mathfrak{H}$ with the metric induced by $D$, which is still denoted by $D$.
The canonical projection from  $\mathfrak{T}$ onto $\mathfrak{H}$
is denoted by $\Pi$, and the volume 
measure $V$ on $\mathfrak{H}$ is the pushforward of the volume measure 
on $\mathfrak{T}$ under $\Pi$. We also introduce 
the ``boundary curve'' $\Lambda=(\Lambda(t))_{t\in\R}$, 
which is simply defined by setting $\Lambda(t)=\Pi(t)$ for every $t\in\R$. 
We note the easy bound $D(a,b)\geq |\ell_a-\ell_b|$, so that the property $\Pi(a)=\Pi(b)$
implies $\ell_a=\ell_b$.

The (random) curve-decorated measure metric space 
$(\mathfrak{H},D,V,\Lambda)$ (or any space having the same distribution)
is called the {\it curve-decorated Brownian half-plane}. It will sometimes be convenient to
consider the pointed measure metric space $(\mathfrak{H},D,V,\Lambda(0))$, to which 
we will refer as the Brownian half-plane. We stress that  $(\mathfrak{H},D)$ is boundedly compact since   any bounded subset of $(\mathfrak{H},D)$ is contained in
a set of the type $\Pi(\{a\in\mathfrak{T}:\ell_a\leq r\})$, which is compact by the transience of labels.
For any fixed $s\in\R$, we can replace $(\Lambda(t))_{t\in\R}$ by $(\Lambda(s+t))_{t\in\R}$ (resp.~$\Lambda(0)$ by $\Lambda(s)$), without changing the
distribution of $(\mathfrak{H},D,V,\Lambda)$ (resp.~of $(\mathfrak{H},D,V,\Lambda(0))$). This translation invariance property is not obvious from
the construction we have given, but follows from other constructions, especially the one
in \cite{BMR,GM}.

 We mention that $(\HH,D)$ is a length space
(this can be deduced from the fact that $D^\circ(a,b)$ is the length of a curve from 
$\Pi(a)$ to $\Pi(b)$ in $\HH$, see e.g.~Section 4.1 in \cite{spine}). 
The space $(\mathfrak{H},D)$ is homeomorphic to the usual closed half-plane, which makes it possible to
define its boundary $\partial\mathfrak{H}$, and $\partial\HH$ is exactly the 
range of $\Lambda$. It follows from the results of \cite{Hausdorff} that the volume measure $V$
coincides with the Hausdorff measure with gauge function $h(r)=c\,r^4\,\log\log(1/r)$, for a
suitable constant $c$. Moreover, $\Lambda$ is a standard boundary curve, meaning 
that the pushforward of Lebesgue measure under $\Lambda$ is the uniform measure on
the boundary, which may be defined by 
$$\langle\nu,\varphi\rangle =\lim_{\ve\to 0} \ve^{-2}\int \varphi(x)\,\mathbf{1}_{\{D(x,\partial \mathfrak{H})<\ve\}}\,V(\dd x),$$
for any bounded continuous function $\varphi$ on $\partial\mathfrak{H}$. In fact, once $\Lambda(0)$ is fixed,
this property characterizes the boundary curve $\Lambda$, up to the replacement of $(\Lambda(t))_{t\in\R}$ by $(\Lambda(-t))_{t\in\R}$.
Finally, we note that  the curve-decorated Brownian half-plane is scale invariant, in the sense
that, for every $\lambda>0$, $(\mathfrak{H},\lambda D,\lambda^4V,\Lambda(\lambda^{-2}\cdot))$
has the same distribution as $(\mathfrak{H},D,V,\Lambda)$.

To simplify notation, we set $\xx=\Lambda(0)=\Pi(0)$, which is viewed as the
distinguished point of $\mathfrak{H}$. It will be important to observe that distances 
from $\xx$ in $\mathfrak{H}$ correspond to labels on the tree $\mathfrak{T}$:
for every $a\in\mathfrak{T}$, we have
$$D(\xx,\Pi(a))=\ell_a,$$
as an easy consequence of formulas \eqref{D1} and \eqref{D2}.
As in Section \ref{curve-deco}, we denote the closed ball of radius $r$ centered at $\xx$ in $\mathfrak{H}$
by $B_r(\mathfrak{H})$. By the previous display, we have
$$B_r(\mathfrak{H})=\Pi(\{a\in\mathfrak{T}:\ell_a\leq r\}),$$
which is  a compact subset of $\mathfrak{H}$ as we already noted.

Let us now turn to hulls. Since  $\mathfrak{H}$ is boundedly compact with the topology of the closed half-plane, it follows that, for any $r>0$, the set $\mathfrak{H} \setminus B_r(\mathfrak{H})$ has  only one unbounded connected component. The hull of radius $r$, denoted by $B^\bullet_r(\mathfrak{H})$, is then defined as  the complement of this unique unbounded connected component.
It is useful 
to characterize the hull in terms of labels on $\mathfrak{T}$. To this end,
for every $a\in\mathfrak{T}\backslash\{0\}$, let 
$m_a$ denote the minimal label on the (unique) geodesic path from
$a$ to infinity in $\mathfrak{T}$ that does not contain $0$: if $a\in \t_{(\omega^i)}$, this geodesic path 
is the concatenation of the ancestral line of $a$ in $\t_{(\omega^i)}$ and the interval $[t_i,\infty)$ (if $t_i>0$) or $(-\infty,t_i]$ (if $t_i<0$).
We also set $m_0=0$. Then we
have
\begin{equation}
\label{formula-hull}
B^\bullet_r(\mathfrak{H})= \Pi(\{a\in \mathfrak{T}:m_a\leq r\}).
\end{equation}
The fact that $m_a>r$ implies that $\Pi(a)\notin B^\bullet_r(\mathfrak{H})$ is easy since
the image under $\Pi$ of the geodesic path from $a$ to $\infty$ in $\mathfrak{T}$ gives a
path from $\Pi(a)$ to $\infty$ in $\mathfrak{H}$ that does not intersect the 
ball $B_r(\mathfrak{H})$. The converse is a consequence of the
so-called cactus bound, which says that any path from $\Pi(a)$ to $\infty$ in
$\mathfrak{H}$ has to visit a point whose distance from $\xx$ is (at most) $m_a$ --- see Proposition 3.1
in \cite{Geodesics} for a version of this result for the Brownian sphere, whose proof is
easily adapted to the present setting.

It follows from \eqref{formula-hull} and the preceding observations that the topological boundary of
$B^\bullet_r(\mathfrak{H})$ can also be written as
\begin{equation}
\label{formu-bdry}
\partial B^\bullet_r(\mathfrak{H})=\Pi\Big(\{a\in\mathfrak{T}: \ell_a=r\hbox{ and }
\ell_b>r\hbox{ for every }b\in G_a\backslash\{a\}\}\Big),
\end{equation}
where we have written $G_a$ for the geodesic path from
$a$ to infinity in $\mathfrak{T}$ that does not contain $0$. In particular, if
$$\beta_r:=\sup\{s\geq 0 : \sqrt{3}\,\mathcal{R}_{-s}\leq r\}\;,\quad \gamma_r:=\sup\{s\geq 0: \sqrt{3}\,\mathcal{R}_s\leq r\},$$
both $\Pi(-\beta_r)$ and $\Pi(\gamma_r)$ belong to $\partial\mathfrak{H}\cap \partial B^\bullet_r(\mathfrak{H})$. In fact,
$\partial B^\bullet_r(\mathfrak{H})$ is the range of a simple path starting at $\Pi(-\beta_r)$
and ending at $\Pi(\gamma_r)$, which does not intersect $\partial\mathfrak{H}$
except at its endpoints. This path is constructed in Lemma \ref{bdry-curve} below, but we first need to introduce some notation.
We consider the exit local times 
$(L^r_s(\omega^i))_{s\geq 0}$ for all $i\in I$ such that $t_i\notin [-\beta_r,\gamma_r]$. 
For every such index $i$, we also set $\alpha_i=\inf\{s\in\R: \ee_s\in\t_{(\omega^i)}\}$.
Finally, for every $s\in\R$, we set
$$L^{\HH,r}_s:=\sum_{i\in I: t_i\notin [-\beta_r,\gamma_r]} L^r_{(s-\alpha_i)^+}(\omega^i),$$
which represents the total exit local time at $r$ accumulated by the exploration process
up to time $s$. We set 
$$Z_r:=L^{\HH,r}_\infty=\sum_{i\in I:t_i\notin [-\beta_r,\gamma_r]} \z_r(\omega^i).$$
We note that $Z_r<\infty$, a.s. Indeed, by \cite[Lemma 4.1]{hull} (see formula \eqref{eq:N:Z:W:*:>0} below),
$\N_x(\z_r\,\mathbf{1}_{\{W_*>0\}})=(r/x)^3$ for every $x>r$, and thus
$$\E[Z_r\mid \mathcal{R}]= 2\int_{(-\infty,-\beta_r)\cup(\gamma_r,\infty)} \frac{r^3}{(\sqrt{3}\mathcal{R}_t)^3}\,\dd t <\infty.$$
Similarly, the sets $\{i\in I:t_i<-\beta_r-\ve,\z_r(\omega^i)>0\}$
and  $\{i\in I:t_i>\gamma_r+\ve,\z_r(\omega^i)>0\}$ are finite, for every $\ve>0$, as a simple consequence of \eqref{hittingpro}. 
On the other hand, it also follows from \eqref{hittingpro} and properties of Bessel processes that the sets $\{i\in I:t_i<-\beta_r,\z_r(\omega^i)>0\}$
and  $\{i\in I:t_i>\gamma_r,\z_r(\omega^i)>0\}$ are both infinite (we use the fact that $\int_{\gamma_r}^{\gamma_r+\ve} (\sqrt{3}\mathcal{R}_t-r)^{-2}\dd r=\infty$,
for every $\ve>0$). 

We write
$Z_r=Z'_r+Z''_r$, where
\begin{equation}\label{def:eq:Z:prime}
Z'_r:=\sum_{i\in I:t_i\in(-\infty,-\beta_r)} \z_r(\omega^i)\;,\quad Z''_r:=\sum_{i\in I:t_i\in(\gamma_r,\infty)} \z_r(\omega^i).
\end{equation}

\begin{lemma}
\label{bdry-curve}
Define $(\kappa(u),u\in [0,Z_r])$ by setting
$$\begin{array}{ll}
\kappa(Z'_r-u):=\inf\{s\in\R: L^{\HH,r}_s\geq u\},\quad&\hbox{if }0<u\leq Z'_r\,,\\
\noalign{\smallskip}
\kappa(Z'_r+u):=\inf\{s\in\R: L^{\HH,r}_s\geq Z_r-u\},&\hbox{if }0\leq u<Z''_r\,,\\
\end{array}
$$
and $\kappa(Z_r)=\kappa(Z_r-)$. Then, $\ee_{\kappa(0)}=-\beta_r$ and $\ee_{\kappa(Z_r)}=\gamma_r$, and the path
$\Lambda^{\bullet,r}$ defined by setting
$\Lambda^{\bullet,r}(u):=\Pi(\ee_{\kappa(u)})$ for every $u\in[0,Z_r]$ is continuous and injective, and
its range is exactly $\partial B^\bullet_r(\HH)$.
\end{lemma}

We will   $Z_r$ as the perimeter of the hull $B^\bullet_r(\HH)$. To understand the intuitive meaning of the definition 
of $\Lambda^{\bullet,r}$, note that $\partial B^\bullet_r(\HH)$ consists of all points $x$ 
with label $r$ such that
there is a continuous path from $x$ to $\infty$ in $\HH$ that visits only points with label greater than $r$ (except at the initial point). 
Leaving aside the points $-\beta_r$ and $\gamma_r$, these are exactly the points $x=\Pi(a)$ such that $a$
belongs to a tree $\t_{(\omega^i)}$ with $t_i\notin[-\beta_r,\gamma_r]$, 
the label of $a$ is $r$, and the ancestral line of $a$ in $\t_{(\omega^i)}$ contains only points with label greater than $r$ ($a$ excepted). The exit local times of the snake trajectories $\omega^i$
with $t_i\notin[-\beta_r,\gamma_r]$ provide a
natural measure on the set of such points $a$, which in turn is used to construct $\Lambda^{\bullet,r}$.

\proof We already noted that the set  $\{i\in I:-\beta_r-\ve<t_i<-\beta_r,\,\z_r(\omega^i)>0\}$ is infinite for every $\ve>0$, and it
easily follows that $\ee_{\kappa(0)}=-\beta_r$. A symmetric argument gives  $\ee_{\kappa(Z_r)}=\gamma_r$.
We then observe that $\Pi(\ee_{\kappa(u)})$ belongs to $\partial B^\bullet_r(\HH)$, for every $u\in[0,Z_r]$.
If $u=0$ or $u=Z_r$, this is immediate from \eqref{formu-bdry}. Then, if $u\in(0,Z_r)$, $\ee_{\kappa(u)}$ belongs to
a tree $\t_{(\omega_i)}$ with $t_i\notin [-\beta_r,\gamma_r]$, and the fact that
$\kappa(u)$ is an increase time of the mapping $s\mapsto L^{\HH,r}_s$ implies
that the label of $\ee_{\kappa(u)}$ is $r$, whereas labels along its ancestral line are greater than $r$. Then we 
use again \eqref{formu-bdry}.

By construction, the mapping $u\mapsto \kappa(u)$ is right-continuous on $[0,Z_r)$.
Then the support property of exit local times implies that, for every $i\in I$ such that
$t_i\notin [-\beta_r,\gamma_r]$, the support of the measure $\dd L^r_s(\omega^i)$ is exactly 
the set $\{s:\tau_r(\omega^i_s)=\zeta_{(\omega^i_s)}\}$. It follows that, a.s. for every $u\in(0,Z_r)$
such that $\ee_{\kappa(u-)}\not=\ee_{\kappa(u)}$, we have $\ell_{\ee_s}>r$ for every $s\in(\kappa(u-),\kappa(u))$, and, 
since we know that $\ell_{\ee_{\kappa(u)}}= \ell_{\ee_{\kappa(u-)}}=r$, the very
definition of $D^\circ$ gives $D^\circ(\ee_{\kappa(u-)},\ee_{\kappa(u)})=0$ and thus $\Pi(\ee_{\kappa(u-)})=\Pi(\ee_{\kappa(u)})$.
This shows that the mapping $u\mapsto \Pi(\ee_{\kappa(u)})$ is continuous. 
The injectivity of this mapping is a consequence of the special Markov property, which implies that $\min\{\ell_{\ee_s}:s\in [\kappa(u),\kappa(v)]\}<r$
(and thus $D^\circ(\ee_{\kappa(u)},\ee_{\kappa(v)})>0 $)
whenever  $0\leq u<v\leq Z_r$. Finally, to verify that any point of $\partial B^\bullet_r(\HH)$ is in the range of $\Lambda^{\bullet,r}$,
we argue as follows. By \eqref{formu-bdry}, a  point of $\partial B^\bullet_r(\HH)$ 
different from both $\Pi(-\beta_r)$ and $\Pi(\gamma_r)$ is of the form $\Pi(a)$, with 
$a\in\t_{(\omega_i)}$ for some $i$ such that $t_i\notin [-\beta_r,\gamma_r]$, and $\ell_a(\omega^i)=r$ and $\ell_b(\omega^i)>r$ for any
other point $b$ of the ancestral line of $a$ in $\t_{(\omega^i)}$. If we write $a=p_{(\omega_i)}(s)$
with $s\in[0,\sigma(\omega^i)]$, the support property of exit local times
ensures that $s$ must be an increase time of $L^r(\omega_i)$, and it follows that $a=\ee_{\kappa(u)}$ for some $u\in(0,Z_r)$ as desired. \endproof

We also observe that
$\{\Lambda^{\bullet,r}(u):u\in(0,Z_r)\}$ does not intersect $\partial \HH$. This follows from the injectivity in the preceding lemma,
since formula \eqref{formu-bdry} shows that $\partial B^\bullet_r(\HH)\cap \partial\HH=\{\Pi(-\beta_r),\Pi(\gamma_r)\}$.
By Jordan's theorem, $B^\bullet_r(\HH)\backslash (\partial B^\bullet_r(\HH)\cup\Pi([-\beta_r,\gamma_r]))$ 
is homeomorphic to the open unit disk, and therefore path-connected. It follows that
$$B^\circ_r(\mathfrak{H}):=B^\bullet_r(\HH)\backslash \partial B^\bullet_r(\HH)$$ is also path-connected.
We shall be interested in the space 
$$\mathfrak{H}_r:=\mathfrak{H}\backslash B^\circ_r(\mathfrak{H}).$$
Again by Jordan's theorem, $\HH_r$ is also homeomorphic to the half-plane, and we can define 
its boundary curve as follows. For every $t\in\R$, we set
\begin{equation}
\label{bd-curve}
\Lambda^r(t):=\left\{\begin{array}{ll}
\Lambda(t-\beta_r)\quad&\hbox{if }t\leq 0,\\
\Lambda^{\bullet,r}(t)&\hbox{if } 0\leq t\leq Z_r,\\
\Lambda(\gamma_r+(t-Z_r))&\hbox{if } t\geq Z_r.
\end{array}
\right.
\end{equation}

\subsection{Peeling the Brownian half-plane}

In this section, we present a version of the spatial Markov property for the (curve-decorated) Brownian half-plane. 
We again fix $r>0$, and consider the space $\mathfrak{H}_r$ defined in the previous section. We write $V_r$ for the restriction 
of the volume measure $V$ to $\HH_r$, and also recall the definition of the boundary curve $\Lambda^r$.
Finally, we write $D_r(x,y)$ for the intrinsic metric on the interior $\mathfrak{H}\backslash B^\bullet_r(\mathfrak{H})$ of $\HH_r$: 
for every $x,y\in \mathfrak{H}\backslash B^\bullet_r(\mathfrak{H})$, $D_r(x,y)$ is the  infimum of lengths of curves 
connecting $x$ to $y$ that stay in $\mathfrak{H}\backslash B^\bullet_r(\mathfrak{H})$ (lengths of course refer to the
distance $D$ on $\HH$).

\begin{theorem}
\label{peeling-HP} 
The intrinsic metric $D_r$ on $\mathfrak{H}\backslash B^\bullet_r(\mathfrak{H})$ has a continuous extension to $\HH_r$
which is a metric on $\HH_r$, and we keep the notation $D_r$ for this metric.
Then the (random) curve-decorated measure metric space 
$(\HH_r,D_r,V_r,\Lambda^r)$ is a curve-decorated Brownian half-plane. 
\end{theorem}

This theorem is an analog of Theorem 22 in \cite{LGR3}, which deals with the Brownian disk, and in fact we will use 
a coupling between the (curve-decorated) Brownian disk and the curve-decorated Brownian half-plane to reduce the proof to this statement.

\proof
By a scaling argument,
it is enough to consider the case $r=1$.
For every $S>0$, write 
$(\D_{(S)},D_{(S)},V_{(S)},\Lambda_{(S)})$ for a curve-decorated free Brownian disk with boundary size $S$ as defined in \cite[Section 4.1]{LGR3} 
(see also \cite{GM}). In view of the coupling with the Brownian half-plane, it will be convenient to make the convention that the
decorating curve $\Lambda_{(S)}$ is indexed by the interval $[-S/2,S/2]$ instead of $[0,S]$: If $\Lambda_{(S)}^\circ$ is the usual
decorating curve indexed by $[0,S]$, this simply means that we take $\Lambda_{(S)}(t)=\Lambda_{(S)}^\circ(t)$ for $t\in[0,S/2]$ and
$\Lambda_{(S)}(t)=\Lambda_{(S)}^\circ(S+t)$ for $t\in[-S/2,0]$. We will make this convention whenever we consider
the curve-decorated free Brownian disk.
We note that $\Lambda_{(S)}(-S/2)=\Lambda_{(S)}(S/2)$ and that
$\Lambda_{(S)}$ is a standard boundary curve, meaning that the 
pushforward of Lebesgue measure on $[-S/2,S/2]$ under $\Lambda_{(S)}$
is the uniform measure on $\partial\D_{(S)}$ (see \cite{LGR3}). The range of $\Lambda_{(S)}$ is the boundary $\partial\D_{(S)}$, and the distinguished point of 
$\D_{(S)}$ is $\xx_{(S)}:=\Lambda_{(S)}(0)$. 

For every
$a>0$ such that $a<D_{(S)}(\xx_{(S)},\Lambda_{(S)}(S/2))$, let $B^\bullet_a(\D_{(S)})$ stand for the hull of radius 
$a$ in $\D_{(S)}$ relative to the point $\Lambda_{(S)}(S/2)=\Lambda_{(S)}(-S/2)$. This means that $\D_{(S)}\backslash B^\bullet_a(\D_{(S)})$
is the connected component containing $\Lambda_{(S)}(S/2)$
of the complement of the closed ball of radius $a$ centered at $\xx_{(S)}$. By convention, if $a\geq D_{(S)}(\xx_{(S)},\Lambda_{(S)}(S/2))$,
we take $B^\bullet_a(\D_{(S)})=\D_{(S)}$.

Let $\delta\in(0,1)$ and $A>10$. By \cite[Proposition 4.2]{GM} (see also \cite[Corollary 3.9]{BMR} or \cite[Lemma 18]{Repre}) and a scaling argument,
we can find $S_0>0$ such that, for every $S\geq S_0$, we can couple the curve-decorated measure metric spaces $(\D_{(S)},D_{(S)},V_{(S)},\Lambda_{(S)})$
and $(\mathfrak{H},D,V,\Lambda)$ in such a way that the following holds with probability at least $1-\delta$: $B^\bullet_A(\D_{(S)})\not=\D_{(S)}$ and there is a measure-preserving 
isometry $\mathfrak{I}$ from $B^\bullet_A(\D_{(S)})$ onto $B^\bullet_A(\HH)$ such that 
\begin{equation}
\label{coup-cur}
\mathfrak{I}(\Lambda_{(S)}(t))= \Lambda(t),\ \hbox{if }
\sup\{s\leq 0:\Lambda_{(S)}(s)\notin B^\bullet_A(\D_{(S)})\}\leq t\leq \inf\{s\geq 0:\Lambda_{(S)}(s)\notin B^\bullet_A(\D_{(S)})\}.
\end{equation}
The preceding properties imply that the isometry $\mathfrak{I}$ maps $B^\bullet_1(\D_{(S)})$ onto $B^\bullet_1(\HH)$ and $\partial B^\bullet_1(\D_{(S)})$ onto $\partial B^\bullet_1(\HH)$.

Note that  \cite[Proposition 4.2]{GM} deals with Brownian disks with a fixed volume, but the result clearly holds also
for free Brownian disks. Moreover, \cite{GM} considers balls instead of hulls, but it is easy to verify
that balls can be replaced by hulls (notice that, both in $\mathbb{H}$ and in $\D_{(S)}$,  the
probability that the hull of radius $\ve$
is contained in the ball of radius $1$ centered at the distinguished point tends to
$1$ when $\ve\to 0$, and use again a scaling argument). 

The existence of the preceding coupling allows us to transfer properties valid in the Brownian
disk to the curve-decorated Brownian half-plane. Let $\wt\D_{(S)}$ be the closure of $\D_{(S)}\backslash B^\bullet_1(\D_{(S)})$, and write $\wt D_{(S)}$ for the
intrinsic metric on $\D_{(S)}\backslash B^\bullet_1(\D_{(S)})$. According to \cite[Theorem 22]{LGR3},  $\wt D_{(S)}$ has a continuous extension 
to $\wt\D_{(S)}$, which is a metric on $\wt\D_{(S)}$. 
In order to get that the intrinsic metric $D_1$ on $\mathfrak{H}\backslash B^\bullet_1(\mathfrak{H})$ has a continuous extension to $\HH_1$,
we need to verify that $D_1(x_n,x_m)$ tends to $0$ as $n,m\to\infty$, for any sequence $(x_n)$ in $\mathfrak{H}\backslash B^\bullet_1(\mathfrak{H})$ 
that converges to a point of $\partial B^\bullet_1(\mathfrak{H})$. However, except on an event of probability at most $\delta$, this 
follows from the preceding coupling and the fact that $\wt D_{(S)}(y_n,y_m)$ tends to $0$ as $n,m\to\infty$, for any sequence $(y_n)$ in $\D_{(S)}\backslash B^\bullet_1(\D_{(S)})$
that converges to a point of $\partial B^\bullet_1(D_{(S)})$. Similarly, we get the fact that the extension of $D_1$ to $\HH_1$
is a metric on $\HH_1$ from the corresponding statement in \cite[Theorem 22]{LGR3}.

From now on, we argue on the event of probability at least $1-\delta$ on which one can define the isometry $\mathfrak{I}$. Recall the definition
of $\beta_1$ in the previous section, and note that
$$-\beta_1=\sup\{s\leq 0: \Lambda(s)\notin B^{\bullet}_1(\HH)\}= \sup\{s\in[-S/2,0]: \Lambda_{(S)}(s)\notin B^{\bullet}_1(\D_{(S)})\},$$
where the second equality follows from \eqref{coup-cur} and the fact that $\mathfrak{I}$ maps $B^\bullet_1(\D_{(S)})$ to $B^\bullet_1(\HH)$.
We then define $\wt\xx_{(S)}=\Lambda_{(S)}(-\beta_1)$, and note that $\mathfrak{I}(\wt\xx_{(S)})=\xx_1:=\Lambda(-\beta_1)=\Lambda^1(0)$,
where $\Lambda^1$ was defined in \eqref{bd-curve}.

By \cite[Theorem 22]{LGR3},
we know that $(\wt\D_{(S)},\wt D_{(S)})$ equipped with the restriction $\wt V_{(S)}$ of the 
volume measure on $\D_{(S)}$ and the distinguished point $\wt\xx_{(S)}$  is a free
Brownian disk with a random boundary size denoted by $\wt Z_{(S)}$, which is pointed at a uniform boundary point. In fact,
as discussed at the end of \cite[Section 4.2]{LGR3}, we can also equip $\wt\D_{(S)}$ with a standard 
boundary curve $(\wt\Lambda_{(S)}(t),t\in [-\wt Z_{(S)}/2,\wt Z_{(S)}/2])$, in such a way that  we have in particular $\wt\Lambda_{(S)}(0)=\wt \xx_{(S)}$ and
$\wt\Lambda_{(S)}(t)=\Lambda_{(S)}(-\beta_1+t)$ for every $t\in [-(S/2)+\beta_1,0]$ such that $t\geq -\wt Z_{(S)}/2$ --- 
note that our convention to index $\wt\Lambda^{(S)}$ by the interval $[-\wt Z_{(S)}/2,\wt Z_{(S)}/2]$ makes the definition of $\wt\Lambda_{(S)}$ look 
different than in \cite{LGR3}. Then, conditionally on $\wt Z_{(S)}$, 
the $4$-tuple $(\wt\D_{(S)},\wt D_{(S)}, \wt V_{(S)},\wt\Lambda_{(S)})$ is a curve-decorated free Brownian disk
with boundary size $\wt Z_{(S)}$.
It is easy to verify 
that the boundary size $\wt Z_{(S)}$ tends to $\infty$ in probability
as $S\to \infty$ (for instance, because $\wt Z_{(S)}\geq (S/2) -\beta_1$ on the event that we are considering).

Next set $a=(A/2)-1$, and write $B_a(\wt\D_{(S)})$ for the closed ball of radius $a$ centered 
at $\wt\xx_{(S)}$ in $\wt\D_{(S)}$.
Note that $B_a(\wt\D_{(S)})\subset B_{A/2}(\D_{(S)})\subset B^\bullet_{A/2}(\D_{(S)})$ (the first inclusion because $D_{(S)}(\wt\xx_{(S)},\xx_{(S)})=1$). If $x$ and $y$
are two points in $B_a(\wt\D_{(S)})\backslash \partial \wt\D_{(S)}$ and $x'=\mathfrak{I}(x)$ and $y'=\mathfrak{I}(y)$ are the 
corresponding points 
of $\mathfrak{H}\backslash B^\bullet_1(\mathfrak{H})$, the intrinsic distance
$\wt D_{(S)}(x,y)$ must coincide 
with  $D_1(x',y')$ --- the point is that a curve from $x'$ to $y'$ that exits $B^\bullet_A(\mathfrak{H})$
must have length greater than $A$, and thus can be disregarded when computing the 
intrinsic distance
between $x'$ and $y'$ (we know that the latter distance is bounded by $2a\leq A-2$, because both $x'$ and $y'$ are at distance at most $a$ from $\xx_1$). 
If $B_a(\HH_1)$ denotes the closed ball of radius $a$
centered at $\xx_1$ in $\HH_1$, we thus
get that $\mathfrak{I}$ induces an isometry from $B_a(\wt\D_{(S)})$ onto $B_a(\HH_1)$,
and  this isometry preserves the volume measures. Furthermore, as in Section \ref{curve-deco}, we 
can also introduce the curve-decorated spaces $\mathfrak{B}_a(\wt\D_{(S)})$ and $\mathfrak{B}_a(\HH_1)$
associated with $B_a(\wt\D_{(S)})$ and $B_a(\HH_1)$ respectively, 
so that, in particular, the decorating curve of $\mathfrak{B}_a(\wt\D_{(S)})$ is
$$\Big( \wt\Lambda_{(S)}(t): -a\vee \sup\{s\leq 0: \wt\Lambda_{(S)}(s)\notin B_a(\wt \D_{(S)})\}\leq t\leq a\wedge \inf\{s\geq 0:\wt\Lambda_{(S)}(s)\notin B_a(\wt \D_{(S)})\}\Big),$$
and the decorating curve of $\mathfrak{B}_a(\HH_1)$ is 
$$\Big( \Lambda^1(t): -a\vee\sup\{s\leq 0: \Lambda^1(s)\notin B_a(\HH_1)\}\leq t\leq a\wedge \inf\{s\geq 0:\Lambda^1(s)\notin B_a(\HH_1)\}\Big).$$
One then verifies that, except on a set of small probability when $S\to\infty$, these two curves are defined on the same interval, and the isometry $\mathfrak{I}$ maps the first one 
to the second one (we omit a few details here).

In conclusion,
one can couple $\mathfrak{H}_1$  and the Brownian disk $\wt\D_{(S)}$ in such a way that the balls $\mathfrak{B}_a(\wt\D_{(S)})$ and $\mathfrak{B}_a(\HH_1)$
coincide except on an event of arbitrarily small probability 
when $S$ is large. Recalling the coupling of \cite[Proposition 4.2]{GM} used in the proof, this also means that we can 
couple $\HH_1$ with a curve-decorated Brownian half-plane $\HH'$ in such a way that the balls of radius $a$ (again viewed as random curve-decorated measure metric spaces)
in both spaces coincide except on an event of arbitrarily small probability. Since this holds for any $a>0$, this suffices 
to prove that $(\HH_1,D_1,V_1,\Lambda^1)$ is a curve-decorated Brownian half-plane. \endproof

We will now show that the space $(\HH_r,D_r,V_r,\Lambda^r)$
in Theorem \ref{peeling-HP} is independent of the hull $B^\bullet_r(\mathfrak{H})$ also viewed as a
curve-decorated measure metric space. We first need to introduce the appropriate metric on $B^\bullet_r(\mathfrak{H})$.
We consider the subset $\kk_r$ of $\mathfrak{T}$
defined by
\begin{equation}
\label{form-hull}
\kk_r= [-\beta_r,\gamma_r]\cup\Big(\bigcup_{i\in I:\,t_i\in[-\beta_r,\gamma_r]}  \t_{(\omega^i)}\Big)\cup\Big(\bigcup_{i\in I:\,t_i\notin[-\beta_r,\gamma_r]} \{a\in \t_{(\omega^i)}:m_a\leq r\}\Big),
\end{equation}
where we recall that $m_a$ stands for the minimal label of $a$ along the geodesic path from $a$ to $\infty$ in $\mathfrak{T}$
that does not contain $0$. It follows from  \eqref{formula-hull} that
$B^\bullet_r(\HH)=\Pi(\kk_r)$. 

We mention the following simple fact. 
Let $a,b\in \kk_r$. Then, in formula \eqref{D1} defining $D^\circ(a,b)$, we may replace the intervals $[a,b]_\infty$
and $[b,a]_\infty$ by $[a,b]_\infty\cap\kk_r$
and $[b,a]_\infty\cap\kk_r$ respectively: the point is that, if the interval $[a,b]_\infty$ contains a point $c\notin \kk_r$,
then, necessarily, it contains another point $c'$ belonging to $\kk_r$ whose label is $r$ and is thus smaller than the label of $c$. 
Informally, the definition of $D^\circ(a,b)$, when $a,b\in \kk_r$ only depends on the labels on $\kk_r$, despite
the fact that the interval $[a,b]_\infty$ may not be contained in $\kk_r$. 

For every $a,b\in\kk_r$, we set
\begin{equation}
\label{pseudo-hull2}
D^{\bullet}_r(a,b) := \mathop{\inf \limits_{a_0,a_1,\ldots,a_p\in \kk_r}}_{a_0=a,\,a_p=b} \sum_{i=1}^p D^{\circ}(a_{i-1},a_i), 
\end{equation}
where the infimum is over all choices of the integer $p\geq 1$ and of the
finite sequence $a_0,a_1,\ldots,a_p$ in $\kk_r$ such that $a_0=a$ and
$a_p=b$. This is similar to the definition \eqref{D1} of $D(a,b)$, but we 
restrict the infimum to ``intermediate'' points $a_1,\ldots,a_{p-1}$ that belong to $\kk_r$. 
Clearly, we have $D(a,b)\leq D^\bullet_r(a,b)\leq D^{\circ}(a,b)$  for every $a,b\in \kk_r$. 
Since  the condition $D(a,b)=0$ can only hold if $D^{\circ}(a,b)=0$, we get that, for every $a,b\in \kk_r$, we have $D^\bullet_r(a,b)=0$ if and only if $D(a,b)=0$. Hence $D^\bullet_r$
induces a metric on $\Pi(\kk_r)=B^{\bullet}_r(\HH)$ and we keep the notation $D^\bullet_r$ for this metric.

Recall that
$B^\circ_r(\mathfrak{H})= B^\bullet_r(\HH)\backslash \partial B^\bullet_r(\HH)$ is the (topological) interior of $B^\bullet_r(\mathfrak{H})$.
Since $B^\circ_r(\mathfrak{H})$ is path-connected, we can define
an intrinsic metric on $B^\circ_r(\mathfrak{H})$. One can then verify that the restriction of $D^\bullet_r$ 
to $B^{\circ}_r(\HH)$ coincides with the intrinsic distance induced by $D$ on $B^{\circ}_r(\HH)$. We omit the details
but refer to the proof of
Proposition 6 in \cite{LGR3} for very similar arguments. 

We finally define a boundary curve for $B^\bullet_r(\HH)$. Recall from the previous section the
definition of the curve $\Lambda^{\bullet,r}=(\Lambda^{\bullet,r}(u))_{u\in(0,Z_r)}$ whose range is $\partial B^\bullet_r(\HH)$.
We set, for every $t\in [0,\gamma_r+\beta_r + Z_r]$,
$$\wh\Lambda^{\bullet,r}(t)=\left\{\begin{array}{ll}
\Lambda(-t)\quad&\hbox{if }0\leq t\leq \beta_r,\\
\Lambda^{\bullet,r}(t-\beta_r)&\hbox{if } \beta_r\leq t\leq \beta_r+Z_r,\\
\Lambda(\gamma_r+\beta_r + Z_r-t)&\hbox{if } \beta_r+Z_r\leq t\leq \gamma_r+\beta_r + Z_r.
\end{array}
\right.
$$
We let $V^{\bullet}_r$ denote the restriction of $V$ to $B^\bullet_r(\HH)$.

\begin{theorem}
\label{indep-peeling}
The random  curve-decorated measure metric spaces $(B^\bullet_r(\HH), D^\bullet_r,V^{\bullet}_r,\wh\Lambda^{\bullet,r})$ and $(\HH_r,D_r,V_r,\Lambda^r)$
are independent. 
\end{theorem}

Theorem \ref{indep-peeling} may be compared to the analogous result for the Brownian disk \cite[Theorem~23]{LGR3}. In the setting
of the Brownian disk, the hull centered at a boundary point (and defined with respect to another distinguished boundary point) and its complement in the disk
become independent only after conditioning on their boundary sizes, and their respective conditional distributions depend only
on their (respective) boundary sizes. In the setting of the Brownian plane, the boundary size of the
hull complement is infinite and no conditioning is needed to get the independence in Theorem \ref{indep-peeling}.

\proof 
Again, we may take $r=1$.
We will derive Theorem \ref{indep-peeling} from \cite[Theorem 23]{LGR3} by the same coupling 
argument that we used to prove Theorem \ref{peeling-HP}, and we keep the notation 
of this proof. In particular, we consider the hull $B^\bullet_1(\D_{(S)})$ on
the event
$\{D_{(S)}(\xx_{(S)},\Lambda_{(S)}(S/2))>1\}$. As it is explained in \cite[Section 6.3]{LGR3},
we can equip $B^\bullet_1(\D_{(S)})$ with an (extended) intrinsic metric defined in a way
similar to the metric $D^\bullet_r$ on $B^\bullet_r(\HH)$ and with the restriction of the 
volume measure on $\D_{(S)}$, and we can also define a decorating curve
on $B^\bullet_1(\D_{(S)})$, which is analogous to $\wh\Lambda^{\bullet,r}$ (see the discussion
before \cite[Theorem 23]{LGR3} for more details). We write
$\mathfrak{B}^\bullet_1(\D_{(S)})$ for the resulting curve-decorated measure metric space. 
Similarly, we write $\mathfrak{B}^\bullet_1(\HH)$ for the curve-decorated measure metric space $(B^\bullet_1(\HH), D^\bullet_1,V^{\bullet}_1,\wh\Lambda^{\bullet,1})$.

Let $A>10$ and $a\in (0,(A/2)-1)$. Also fix $\delta>0$. As in the proof of Theorem \ref{peeling-HP}, 
for every large enough $S$, we can couple $(\D_{(S)},D_{(S)},V_{(S)},\Lambda_{(S)})$
and $(\mathfrak{H},D,V,\Lambda)$ in such a way that, except on a set of probability at most $\delta$,
there is a measure-preserving 
isometry $\mathfrak{I}$ from $B^\bullet_A(\D_{(S)})$ onto $B^\bullet_A(\HH)$ such that 
\eqref{coup-cur} holds. Then it is not hard to verify that $\mathfrak{I}$
induces an isometry from $\mathfrak{B}^\bullet_1(\D_{(S)})$ onto $\mathfrak{B}^\bullet_1(\HH)$,
which preserves the volume measures and the decorating curves. As in the proof of Theorem \ref{peeling-HP}, we 
also know that, except on a set of small probability when
$S$ is large,  $\mathfrak{I}$
induces an isometry from $\mathfrak{B}_a(\wt\D_{(S)})$ onto $\mathfrak{B}_a(\HH_1)$. 
Summarizing, if  $S$ is large enough, 
except on event of small probability, we
can couple $(\D_{(S)},D_{(S)},V_{(S)},\Lambda_{(S)})$
and $(\mathfrak{H},D,V,\Lambda)$ so that $(\mathfrak{B}^\bullet_1(\D_{(S)}),\mathfrak{B}_a(\wt\D_{(S)}))=
(\mathfrak{B}^\bullet_1(\HH),\mathfrak{B}_a(\HH_1))$, where the equality is in the sense of
isometry between curve-decorated measure metric spaces. 

Let $F$ and $G$ be two bounded measurable functions defined on the space $\M^{\mathrm{GHPU}}$ of all curve-decorated compact measure metric spaces.
It follows from the preceding considerations that
\begin{equation}
\label{Mar4}
\Big| \E\Big[F(\mathfrak{B}_a(\wt \D_{(S)}))\,G(\mathfrak{B}^\bullet_1(\D_{(S)}))\,\mathbf{1}_{\{D_{(S)}(\xx_{(S)},\Lambda_{(S)}(S/2))>1\}}\Big]
- \E\Big[F(\mathfrak{B}_a(\HH_1))\,G(\mathfrak{B}^\bullet_1(\HH))\Big]\Big|\build{\la}_{S\to\infty}^{} 0.
\end{equation}
On the other hand, Theorems 22 and 23 in \cite{LGR3} imply that
$$
\E\Big[F(\wt \D_{(S)})\,G(\mathfrak{B}^\bullet_1(\D_{(S)}))\,\mathbf{1}_{\{D_{(S)}(\xx_{(S)},\Lambda_{(S)}(S/2))>1\}}\Big]
=\E\Big[\Theta_{\wt Z_{(S)}}(F)\,G(\mathfrak{B}^\bullet_1(\D_{(S)}))\,\mathbf{1}_{\{D_{(S)}(\xx_{(S)},\Lambda_{(S)}(S/2))>1\}}\Big]
$$
where $\wt \D_{(S)}$ is also viewed as a curve-decorated measure metric space (as in the proof of Theorem \ref{peeling-HP}) and we write $\Theta_z$ for the distribution of the
(curve-decorated) free Brownian disk with perimeter $z$. We can specialize this 
equality to the case where $F$ only depends on the ball of radius $a$. It follows that
\begin{align*}
&\E\Big[F(\mathfrak{B}_a(\wt \D_{(S)}))\,G(\mathfrak{B}^\bullet_1(\D_{(S)}))\,\mathbf{1}_{\{D_{(S)}(\xx_{(S)},\Lambda_{(S)}(S/2))>1\}}\Big]\\
&\quad=\E\Big[\Theta_{\wt Z_{(S)}}(F\circ\mathfrak{B}_a)\,G(\mathfrak{B}^\bullet_1(\D_{(S)}))\,\mathbf{1}_{\{D_{(S)}(\xx_{(S)},\Lambda_{(S)}(S/2))>1\}}\Big].
\end{align*}
The coupling between the curve-decorated Brownian half-plane and the free Brownian disk of perimeter $z$ ensures that
$$\Theta_z(F\circ \mathfrak{B}_a) \build\la_{z\to\infty}^{} \Theta_\infty(F\circ \mathfrak{B}_a),$$
where $\Theta_\infty$ is the distribution of the curve-decorated Brownian half-plane. Since $\wt Z_{(S)}$ tends to $\infty$
as $S\to\infty$,
the last two displays imply that
\begin{align}
\label{Mar3}
&\Big| \E\Big[F(\mathfrak{B}_a(\wt \D_{(S)}))\,G(\mathfrak{B}^\bullet_1(\D_{(S)}))\,\mathbf{1}_{\{D_{(S)}(\xx_{(S)},\Lambda_{(S)}(S/2))>1\}}\Big]\nonumber\\
&\quad- \Theta_\infty(F\circ\mathfrak{B}_a)\,\E\Big[G(\mathfrak{B}^\bullet_1(\D_{(S)}))\,\mathbf{1}_{\{D_{(S)}(\xx_{(S)},\Lambda_{(S)}(S/2))>1\}}\Big]\Big|
\build{\la}_{S\to\infty}^{} 0.
\end{align}
Using \eqref{Mar4} twice (the second time with $F=1$), we deduce from \eqref{Mar3} that
$$\E\Big[F(\mathfrak{B}_a(\HH_1))\,G(\mathfrak{B}^\bullet_1(\HH))\Big]=\Theta_\infty(F\circ\mathfrak{B}_a)\,\E\Big[G(\mathfrak{B}^\bullet_1(\HH))\Big].$$
This gives the desired independence property. \endproof

\section{Some explicit formulas}
\label{sec:explicit}

In this section, we provide explicit formulas for the joint distribution of the variables
$\beta_r,\gamma_r, Z_r$, which determine the boundary size of the hull $B^\bullet_r(\HH)$, and 
of the volume of the hull $B^\bullet_r(\HH)$ (see Fig.~2). To simplify notation, we write $\v_r:=V(B^\bullet_r(\HH))$
for the latter volume.

 \begin{figure}[!h]
\centering
   \includegraphics[width=12cm]{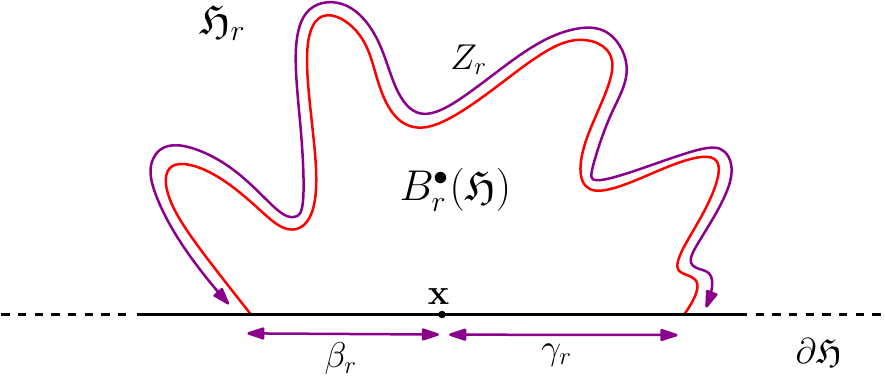}
    \caption{Illustration of the boundary lengths $\beta_r$, $\gamma_r$ and $Z_r$. }
  \end{figure}

\begin{proposition}
\label{formulas!}
The random variables $\beta_r,\gamma_r$ and $Z_r$
are independent. Moreover, $\beta_r$ and $\gamma_r$
have the same distribution, whose density is given by
$$t\mapsto \frac{3^{-3/2}r^3}{\sqrt{2\pi}}\,t^{-5/2}\,\exp(-\frac{r^2}{6t}),$$
and $Z_r$ is exponentially distributed with mean $2r^2/3$. The joint distribution
of $(\beta_r,\gamma_r,Z_r,\v_r)$ is given by the formula
\begin{equation}
\label{joint-Lap}
\mathbb{E}\Big[\exp\big(-\lambda Z_{r}-\nu_1 \beta_r-\nu_2 \gamma_r-\mu \v_{r}\big)\Big]=\frac{G_r(\mu,\nu_1)\,G_r(\mu,\nu_2)}{\frac{2}{3}\lambda\,r^2+\sqrt{2\mu}\,r^2\big(\coth((2\mu)^{\frac{1}{4}}r)^{2}-\frac{2}{3}\big)},
\end{equation}
where $G_r$ is the function:
\begin{equation}
\label{formu-G}
G_r(\mu,\nu):= r \exp\Big(-r\,\sqrt{\frac{2}{3}} \sqrt{\sqrt{2\mu}+\nu}\Big)\times\Big((2\mu)^{\frac{1}{4}}\coth((2\mu)^{\frac{1}{4}}r)+\sqrt{\frac{2}{3}}\sqrt{\sqrt{2\mu}+\nu}\,\Big).
\end{equation}
\end{proposition}
We have not been able to find an intuitive explanation for the exponential distribution of $Z_r$. Interestingly, in the
Brownian plane, the boundary size of the hull follows a $\Gamma(3/2)$-distribution \cite[Proposition~1.2]{hull}.

\proof Recall the definition  of the variables $Z'_r$ and $Z''_r$ in \eqref{def:eq:Z:prime}, such that $Z'_r+Z''_r=Z_r$. From the independence 
of the processes $(\mathcal{R}_t)_{t\leq 0}$ and $(\mathcal{R}_t)_{t\geq0}$ one immediately gets that the
pairs $(\beta_r,Z'_r)$ and $(\gamma_r,Z''_r)$ are independent and identically distributed. Furthermore, using 
the well-known fact that $(\mathcal{R}_{\gamma_r+t})_{t\geq 0}$ is independent of $\gamma_r$, one
obtains that $Z''_r$ is independent of $\gamma_r$, and similarly $Z'_r$ is independent of $\beta_r$. This discussion
shows that $\beta_r,\gamma_r$ and $Z_r$ are independent. 

By definition, $\gamma_r$ and $\beta_r$ are distributed as the last hitting time of $r/\sqrt{3}$ by a five-dimensional
Bessel process started from $0$, and their density is well known \cite{Get} to be as stated in the proposition. 
We note that the Laplace transform of $\gamma_r$ (or $\beta_r$) is given by
$$\E\big[e^{-\lambda \gamma_r}\big]= \Big(1+ r\sqrt{\frac{2\lambda}{3}}\Big)\,\exp\Big(-r \sqrt{\frac{2\lambda}{3}}\Big).$$

Let us turn to the distribution of $Z_r$.  Since, conditionally on $\mathcal{R}$, the point measure  $\mathcal{N}(\dd t\,\dd\omega)$  is Poisson 
with intensity $2\,\mathbf{1}_{\{W_*(\omega)>0\}}\,\dd t\,\N_{\sqrt{3}\mathcal{R}_t}(\dd \omega)$, we have
$$\E\big[e^{-\lambda Z''_r}\big]= \E\Big[\exp\Big(-2\int_{\gamma_r}^\infty \dd t\,\N_{\sqrt{3}\mathcal{R}_t}\Big((1-e^{-\lambda\z_r})\mathbf{1}_{\{W_*>0\}}\Big)\Big)\Big].$$
From \cite[Lemma 4.1]{hull}, we have, for every $x>r$,
\begin{equation}\label{eq:N:Z:W:*:>0}
\N_x\Big((1-e^{-\lambda\z_r})\mathbf{1}_{\{W_*>0\}}\Big)=\frac{3}{2}\Big(\Big(x-r+ (\frac{2\lambda}{3} +r^{-2})^{-1/2}\Big)^{-2}-x^{-2}\Big).
\end{equation}
It follows, that, for $t>\gamma_r=\sup\{s\geq 0:\mathcal{R}_s\leq r/\sqrt{3}\}$,
\begin{equation}
\label{formu-13}
\N_{\sqrt{3}\mathcal{R}_t}\Big((1-e^{-\lambda\z_r})\mathbf{1}_{\{W_*>0\}}\Big)=\frac{1}{2}\Big((\mathcal{R}_t-b)^{-2}-(\mathcal{R}_t)^{-2}\Big),
\end{equation}
where we have set
$$b=\frac{r}{\sqrt{3}}-(2\lambda +3r^{-2})^{-1/2} \in(0,\frac{r}{\sqrt{3}}).$$
\begin{lemma}
\label{tec-formu}
For every $x>0$, set $\mathcal{L}_x:=\sup\{t\geq 0:\mathcal{R}_t\leq x\}$. Then, for every $c<x<y$, 
$$\E\Big[\exp\Big(-\int_{\mathcal{L}_x}^{\mathcal{L}_y} \dd t\, \Big((\mathcal{R}_t-c)^{-2}-(\mathcal{R}_t)^{-2}\Big)\Big)\Big] =\frac{y}{x}\times \frac{x-c}{y-c}.$$
\end{lemma}

The proof is analogous to the proof of Lemma 4.2 in \cite{hull} and is deferred to the Appendix. We apply Lemma \ref{tec-formu}
with $x=r/\sqrt{3}$, $c=b$ and taking limits as $y\to\infty$. It follows that
$$\E[e^{-\lambda Z''_r}]=\frac{(r/\sqrt{3})-b}{r/\sqrt{3}}= \big(1+\frac{2\lambda r^2}{3}\big)^{-1/2}.$$
Since $Z_r=Z'_r+Z''_r$, and $Z'_r$ and $Z''_r$ are independent and identically distributed, we finally get
$$\E[e^{-\lambda Z_r}]=\big(1+\frac{2\lambda r^2}{3}\big)^{-1},$$
which gives the desired distribution of $Z_r$. 

Let us turn to the proof of \eqref{joint-Lap}. If $\w$ is a stopped path, we write $\mathrm{min}(\w):=\min\{\w(t):0\leq t\leq \zeta_{(\w)}\}$.
Then, from formula \eqref{form-hull} and the definition of the volume measure, we get that
$$\v_r=\v^0_r + \v^1_r + \v^2_r,$$
where
\begin{align*}
\v^0_r&:=\sum_{i\in I:t_i\in (-\infty,-\beta_r)\cup (\gamma_r,\infty)} \int_0^{\sigma(\omega^i)} \dd s\,\mathbf{1}_{\{\mathrm{min}(\omega^i_s)\leq r\}},\\
\v^1_r&:=\sum_{i \in I: t_i\in [-\beta_r,0]}\sigma(\omega^i),\\
\v^2_r&:=\sum_{i \in I: t_i\in[0,\gamma_r]}\sigma(\omega^i).
\end{align*}
Using the independence of the processes $(\mathcal{R}_t)_{t\geq 0}$ and $(\mathcal{R}_t)_{t\leq 0}$, the fact that
$(\mathcal{R}_{\gamma_r+t})_{t\geq 0}$ is independent of $(\mathcal{R}_t)_{0\leq t\leq \gamma_r}$ (and the analogous property
for $(\mathcal{R}_t)_{t\leq 0}$) and properties of Poisson measures, one immediately verifies that the three pairs 
$(Z_r,\v^0_r)$, $(\beta_r,\v^1_r)$ and $(\gamma_r,\v^2_r)$ are independent, and moreover the pairs
$(\beta_r,\v^1_r)$ and $(\gamma_r,\v^2_r)$ have the same distribution. 

Let us start by discussing the pair $(\gamma_r,\v^2_r)$. For every $\mu>0$, we have
$$\E\Big[\exp(-\mu \v^2_r)\,\Big|\, (\mathcal{R}_t)_{t\geq 0}\Big]=\exp\Big(-2\int_0^{\gamma_r} \dd t\,\N_{\sqrt{3}\mathcal{R}_t}\Big((1-e^{-\mu\sigma})\,\mathbf{1}_{\{W_*>0\}}\Big)\Big). $$
For every $x>0$, set 
$$g_\mu(x)=\N_x\Big((1-e^{-\mu\sigma})\,\mathbf{1}_{\{W_*>0\}}\Big)=\N_x\Big(1-\mathbf{1}_{\{W_*>0\}}\,e^{-\mu\sigma}\Big) -\frac{3}{2x^2}.$$
By \cite[Lemma 7]{Del}, we have
$$g_\mu(x)=\sqrt{\frac{\mu}{2}}\Big(3\,\coth\big((2\mu)^{1/4}x\big)^2 - 2\Big)  -\frac{3}{2x^2}.$$
\begin{lemma}
\label{tec-formu2}
For every $\mu,\nu >0$, we have
$$\E\Big[\exp\Big(-\nu \gamma_r - 2 \int_0^{\gamma_r} \dd t\,g_\mu(\sqrt{3}\mathcal{R}_t)\Big)\Big] = G_r(\mu,\nu),$$
where the function $G_r(\mu,\nu)$ is defined in \eqref{formu-G}.
\end{lemma}

We postpone the proof of Lemma \ref{tec-formu2} until the Appendix and complete the proof of Proposition \ref{formulas!}. It readily follows 
from Lemma \ref{tec-formu2} that
\begin{equation}
\label{formula1}
\E[\exp(-\nu \gamma_r - \mu \v^2_r)]= \E\Big[\exp\Big(-\nu \gamma_r - 2 \int_0^{\gamma_r} \dd t\,g_\mu(\sqrt{3}\mathcal{R}_t)\Big)\Big] = G_r(\mu,\nu).
\end{equation}
The same formula obviously holds if the pair $(\gamma_r,\v^2_r)$ is replaced by $(\beta_r,\v^1_r)$.
It remains to compute $\E[\exp(-\lambda Z_r - \mu \v^0_r)]$. An application of the special Markov property gives
$$\E[e^{-\mu \v^0_r}\mid Z_r] = \exp\Big(-Z_r\,\N_r\Big((1-e^{-\mu\sigma})\,\mathbf{1}_{\{W_*>0\}}\Big)\Big)=\exp(-g_\mu(r)\,Z_r).$$
Hence, recalling the exponential distribution of $Z_r$
\begin{align}
\label{formula2}
\E[\exp(-\lambda Z_r - \mu \v^0_r)]&=\E[\exp(-(\lambda+g_\mu(r)) Z_r)]\nonumber\\
&=\frac{3}{2r^2}\int_0^\infty \dd t\,\exp\Big(-(\lambda+g_\mu(r))t- \frac{3}{2r^2}t\Big)\nonumber\\
&=\frac{3}{2r^2}\times \frac{1}{\lambda + \sqrt{\frac{\mu}{2}}\Big(3\,\coth\big((2\mu)^{1/4}r\big)^2 - 2\Big)}\,.
\end{align}
Finally, using the independence of the three pairs 
$(Z_r,\v^0_r)$, $(\beta_r,\v^1_r)$ and $(\gamma_r,\v^2_r)$ and formulas \eqref{formula1} and \eqref{formula2}, we get
\begin{align*}
&\mathbb{E}\big[\exp\big(-\lambda Z_{r}-\nu_1 \beta_r-\nu_2 \gamma_r-\mu \v_{r}\big)\big]\\
&\quad=
\E[\exp(-\lambda Z_r - \mu \v^0_r)] \times \E[\exp(-\nu_1 \beta_r - \mu \v^1_r)]\times \E[\exp(-\nu_2 \gamma_r - \mu \v^2_r)]\\
&\quad=\frac{3}{2r^2}\times \frac{1}{\lambda + \sqrt{\frac{\mu}{2}}\Big(3\,\coth\big((2\mu)^{1/4}r\big)^2 - 2\Big)}\times G_r(\mu,\nu_1)\times G_r(\mu,\nu_2).
\end{align*}
This completes the proof of Proposition \ref{formulas!}.
\endproof

\section{The perimeter process}
\label{sec:peripro}

In this section, we study the process $(Z_t)_{t \geq 0}$, where, by convention, we take $Z_0 := 0$.  Our first goal is to compute the finite-dimensional marginals of $(Z_t)_{t \geq 0}$. To this end, for every $t\geq 0$, we consider  the $\sigma$-field $\mathcal{G}_t$ generated by the processes $(\mathcal{R}_{\gamma_t + r})_{r \geq 0}$ and $(\mathcal{R}_{-\beta_t - r})_{r \geq 0}$, the point measure:
$$\sum_{i\in I, t_i\notin [-\beta_t,\gamma_t]} \delta_{(t_i,\omega^i)},$$
and the $\mathbb{P}$-negligible sets.  Then $(\mathcal{G}_{t})_{t\geq 0}$ is a backward filtration
(meaning that $\mathcal{G}_t\subset \mathcal{G}_s$ if $s\leq t$).  We also observe that $Z_t$ is $\mathcal{G}_{t}$-measurable, since $Z_t=\sum_{i\in I, t_i\notin [-\beta_t,\gamma_t]} \mathcal{Z}_t(\omega^i)$. 
\begin{lemma}\label{eq:two:points}
Let $0\leq s\leq t$. Then, for every $\lambda>0$, we have:
%\begin{align*}
$$\mathbb{E}\Big[\exp\big(-\lambda Z_s\big) \Big| ~\mathcal{G}_t\Big]
=\Bigg(\frac{t}{s+(t-s) \big(1+ \frac{2\lambda s^2}{3}\big)^{1/2}}\Bigg)^{2} \exp\Bigg(- \frac{3Z_t}{2}\Bigg( \frac{1}{\Big(t-s+(\frac{2\lambda}{3}+ s^{-2})^{-1/2}\Big)^2}-\frac{1}{t^2}\Bigg) \Bigg).
$$
\end{lemma}
In particular, for every fixed $t>0$, $(Z_{t-s})_{s\in[0,t)}$ is a time-inhomogeneous Markov process with respect to the filtration $(\mathcal{G}_{t-s})_{s\in[0,t)}$. Together
with the distribution of $Z_t$ obtained in Proposition \ref{formulas!}, this characterizes the finite-dimensional marginals of of $(Z_t)_{t \geq 0}$. We mention that 
there is a striking analogy between Lemma \ref{eq:two:points} and \cite[Proposition 4.3]{hull}, which was 
the similar result for
 the Brownian plane. 
\begin{proof}
Let $0< s\leq t$. 
We have $Z_s = Z^0_s + Z^1_s + Z^2_s,$ where
\begin{align*}
Z^0_s&:=\sum_{i\in I:t_i\in (-\infty,-\beta_t)\cup (\gamma_t,\infty)} \mathcal{Z}_{s}(\omega^i),\\
Z^1_s&:=\sum_{i \in I: t_i\in [-\beta_t,-\beta_s)}\mathcal{Z}_s(\omega^i),\\
Z^2_s&:=\sum_{i \in I: t_i\in(\gamma_s,\gamma_t]}\mathcal{Z}_s(\omega^i).
\end{align*}
The argument is now similar to the proof of Proposition \ref{formulas!}.  By construction, $Z_{s}^{1}$ and $Z_s^{2}$ are independent of $\mathcal{G}_t$ (hence of $Z^0_s$) and have the same distribution. Furthermore, from \eqref{formu-13}, we have
\begin{align*}
\mathbb{E}\Big[\exp\big(-\lambda Z_s^{2}\big)\Big]&= \mathbb{E}\Big[\exp\Big(-2\int_{\gamma_s}^{\gamma_t} \mathrm{d} r~\N_{\sqrt{3}\mathcal{R}_r}\Big((1-e^{-\lambda\z_s})\mathbf{1}_{\{W_*>0\}}\Big)\Big)\Big]\\
&= \mathbb{E}\Big[\exp\Big(-\int_{\gamma_s}^{\gamma_t}\mathrm{d} r\,\Big((\mathcal{R}_r-b)^{-2}-(\mathcal{R}_r)^{-2}\Big)\Big)\Big]
\end{align*}
where $b=(s/\sqrt{3})-(2\lambda +3s^{-2})^{-1/2}.$ Hence an application of Lemma \ref{tec-formu} gives:
$$\mathbb{E}\Big[\exp\big(-\lambda Z_s^{2}\big)\Big]=\frac{t}{s+(t-s) \big(1+ \frac{2\lambda s^2}{3}\big)^{\frac{1}{2}}}.$$
Finally, by the special Markov property,
\begin{align*}
\mathbb{E}\Big[\exp\big(-\lambda Z_s^{0}\big) \Big| ~\mathcal{G}_t\Big]&=\exp\Big(-Z_t\,\N_{t}\Big((1-e^{-\lambda\z_s})\mathbf{1}_{\{W_*>0\}}\Big)\Big)\\
&= \exp\Bigg(- \frac{3Z_t}{2}\Bigg( \frac{1}{(t-s+(\frac{2\lambda}{3}+ s^{-2})^{-\frac{1}{2}})^2}-\frac{1}{t^2}\Bigg) \Bigg),
\end{align*}
using  \eqref{eq:N:Z:W:*:>0} in the second equality. This  completes the proof of the lemma.
\end{proof}

Our aim now is to derive  two distinct characterizations of the process $(Z_t)_{t \geq 0}$ from Lemma \ref{eq:two:points}. More precisely, we will prove  that:
\begin{itemize}
\item viewed backward in time, the process $(Z_t)_{t \geq 0}$ is an $h$-transform of a continuous-state branching process with immigration;
\item viewed in the usual forward time direction, $(Z_t)_{t \geq 0}$ is a self-similar Markov process starting from $0$, which can be characterized by the associated L\'evy process.
\end{itemize}
These characterizations are analogous to those provided for the Brownian plane in \cite[Proposition 4.4]{hull} and \cite[Section 11.2]{growth}. We start with the first characterization and we will identify the transition kernel whose Laplace transform is given by Lemma \ref{eq:two:points}. We let $\D(\R_+,\R)$ stand for the Skorokhod space  of c\`adl\`ag functions 
from $\R_+$ into $\R$. We write $(Y_t)_{t\geq 0}$ for the canonical process 
on $\D(\R_+,\R)$, and $(\f_t)_{t\geq 0}$ for the canonical filtration. We then define, for every $x \geq 0$, the probability measure $P_x$ as the law of the continuous-state branching process with immigration, with branching mechanism $\Psi(\lambda) = \sqrt{8/3} \, \lambda^{3/2}$ and  immigration mechanism $H(\lambda) = \sqrt{8/3}\, \lambda^{1/2}$.
We refer to \cite{KW} for the general theory of continuous-state branching processes with immigration (see also the survey \cite{Li}). We have,
for every $s\geq 0$,
\begin{equation}\label{eq:Y:im}
E_x\big[\exp(-\lambda Y_s)\big] = \exp\Big( -x\,u_\lambda(s)  - \int_{0}^{s} H(u_\lambda(r)) \,\dd r\Big) = \Big(1+s\sqrt{2\lambda/3}\, \Big)^{-2} \exp\big( -x\,u_{\lambda}(s)  \big),
\end{equation}
where $u_\lambda(s):=(\lambda^{-1/2}+s\sqrt{2/3})^{-2}$ solves
$\frac{\mathrm{d} u_\lambda(s)}{\mathrm{d}s}=-\Psi(u_\lambda(s))$, with $u_0(\lambda)=\lambda$.

For every $a>0$, we set 
$$h_a(x):= a^{-2} \exp\big(-\frac{3}{2 a^2} x\big).$$
Then \eqref{eq:Y:im} shows that, for every $s\geq 0$, 
$$E_x\Big[h_{a}( Y_s)\Big]= h_{a+s}(x).$$
Let us fix $t>0$. It follows from the last display that the process $(h_{t-r}(Y_r))_{r\in[0,t)}$ is a martingale under $P_x$. Hence, for every $x>0$, we may define a 
probability measure $P_x^{(t)}$ on $\D([0,t),\R)$, the Skorokhod space  of c\`adl\`ag functions 
from $[0,t)$ into $\R$, by requiring that, for every $r\in [0,t)$,
\begin{equation}
\label{htrans}
{\frac{\dd P^{(t)}_x}{\dd P_x}}_{\Big|\f_r} = \frac{h_{t-r}(Y_{r})}{h_t(x)},
\end{equation}
where we slighly abuse notation by viewing $P_x$ as a probability measure on $\D([0,t),\R)$ and keeping the same notation
$Y_r$ and $\mathcal{F}_r$ for the canonical process and the canonical filtration on $\D([0,t),\R)$. The process
$$\Big(\frac{1}{h_{t-r}(Y_{r})}\Big)_{r\in [0,t)}$$
is a (nonnegative) martingale under $P^{(t)}_x$ and thus must converge $P^{(t)}_x$ a.s. as $r\uparrow t$
to a finite limit. Clearly, this is only possible if $Y_r$ converges to $0$
as $r\uparrow t$, $P^{(t)}_x$ a.s. It follows from these considerations that we can also view
$P^{(t)}_x$ as a probability measure on $\D(\R_+,\R)$, which is supported on $\{Y_s=0\hbox{ for every }s\geq t\}$.

\begin{proposition}\label{prop:Z:immi}
Let $t>0$ and $x>0$. Conditionally on $Z_t=x$, the process $(Z_{t-r})_{r\leq  t}$
has the same finite-dimensional marginals as the process $(Y_{r})_{r\leq t}$ under $P_x^{(t)}$.
\end{proposition}

\begin{proof}
Since  $(Z_{t-r})_{r\in[0,t)}$ (under $\P$) and $(Y_r)_{r\in[0,t)}$ (under $P_x^{(t)}$) are both time-inhomogeneous 
Markov processes, it is enough to verify that they have the same transition kernels. 
For $t>0$ and $y\geq 0$, let $q_t(y,\mathrm{d} z)$ be the distribution of $Y_t$ under $P_y$, and, for every $0\leq r<s<t$,
$$
\pi_{r,s}^{(t)}(y,\mathrm{d}z) :=\frac{h_{t-s}(z)}{h_{t-r}(y)} q_{s-r}(y,\mathrm{d} z),
$$
Then  the kernels $\pi_{r,s}^{(t)}(y,\mathrm{d}z) $ are the transition kernels of $(Y_{r} )_{r\in[0,t)}$ under $P_x^{(t)}$. Next, for every $0\leq r<s<t$ and $\lambda\geq 0$, we have
\begin{align*}
\int \pi^{(t)}_{r,s}(z, \mathrm{d} y) \exp\big(-\lambda y\big)&=\Big(\frac{t-r}{t-s}\Big)^{2}\exp\Big(\frac{3}{2(t-r)^2} \,z\Big)\, E_z\Big[\exp\Big(-(\lambda+\frac{3}{2(t-s)^{2}}) Y_{s-r} \Big) \Big]
\end{align*}
and then an application of \eqref{eq:Y:im} shows that the right-hand side is  equal to
\begin{align*}
&\Bigg(\frac{t-r}{t-s+(s-r) \big(1+\frac{2\lambda (t-s)^2}{3}\big)^{1/2}}\Bigg)^{2} \\
&\quad\times \exp\Bigg(-\frac{3z}{2}\Bigg(\frac{1}{\big( s-r+ \big(\frac{2\lambda}{3} + (t-s)^{-2} \big)^{-1/2} \big)^2}-\frac{1}{\big(t-r\big)^2}\Bigg)\Bigg) .
\end{align*}
Comparing with Lemma \ref{eq:two:points}, we get
$$\mathbb{E}\Big[\exp(-\lambda Z_{t-s})\,\Big|\, \mathcal{G}_{t-r} \Big]= \int \pi^{(t)}_{r,s}(Z_{t-r}, \mathrm{d} y) \exp\big(-\lambda y\big),$$
which completes the proof of the proposition. 
\end{proof}
Proposition \ref{prop:Z:immi} entails that the process $(Z_t)_{t\geq 0}$  possesses a right-continuous modification, which we consider from now on. We turn to the second characterization of $(Z_t)_{t \geq 0}$, as a self-similar Markov process. Let $\psi$ be the function defined for $q>0$ by
$$\psi(q):=\sqrt{\frac{8}{3}}\,q\,\frac{\Gamma(q+1)}{\Gamma(q+\frac{1}{2})}.$$
Then $\psi$ is the Laplace exponent of  a spectrally negative L\'evy process $(\xi_t)_{t\geq 0}$ (meaning that $\E[\exp(\lambda \xi_t)]=\exp(t\psi(\lambda))$
for $\lambda\geq 0$). In fact $\xi$ belongs to the class of hypergeometric L\'evy processes discussed in Chapter 4 of \cite{KP}. To be precise, the 
spectrally negative case is a borderline case excluded in Theorem~4.6 of \cite{KP}, but one can still use Proposition 4.1 and Theorem 4.4 of \cite{KP} 
to get the L\'evy-Khintchine representation
$$\psi(q)=4\sqrt{\frac{2}{3\pi}}q+\frac{1}{\sqrt{6\pi}}\int_{-\infty}^{0}(\exp(q y)-1+q(1-\exp(y)))\frac{2+\exp(y)}{(1-\exp(y))^{5/2}}\exp(y)\mathrm{d} y. $$

For every $x>0$, we then write $P^\uparrow_x$ for the distribution (on the Skorokhod space $\D(\R_+,\R)$) of the self-similar Markov process with index $1/2$ and initial value $x$, which is associated 
with the L\'evy process $\xi$ via the Lamperti transformation. In other words, $P^\uparrow_x$ is the law of the process  $(x\exp(\xi_{\alpha(t)}))_{t \geq 0}$, where,
for every $t\geq 0$,  $\alpha(t) := \inf\{s \geq 0 : \sqrt{x} \int_{0}^{s} \exp(\frac{1}{2} \xi_{r}) \mathrm{d}r \geq t\}$. The self-similarity property means that, for every $\lambda>0$,
the law of $(\sqrt{\lambda}\,Y_{\lambda t})_{t\geq 0}$ under $P^\uparrow_x$ is $P^\uparrow_{x\sqrt{\lambda}}$.

Since $\xi$ has no positive jumps, we can apply \cite[Proposition 1]{BY}, which shows that $P^\uparrow_x$ converges weakly as $x\downarrow 0$ to a limiting law denoted by $P^\uparrow_0$, which is characterized by the following two properties holding for every $t>0$:
\begin{itemize}
\item[(i)] the law of $Y_t$ under $P^\uparrow_0$ is exponential with mean $2t^2/3$;
\item[(ii)] under $P^\uparrow_0$ and conditionally on $(Y_r)_{0 \leq r \leq t}$, the process $(Y_{t+s})_{s \geq 0}$ is distributed according to $P^\uparrow_{Y_t}$.
\end{itemize}
In particular, property (i) follows from the formula for the moments $E^\uparrow_0[(Y_t)^k]$ found in \cite[Proposition 1]{BY}. At this stage, we note
that \cite{BY} deals with self-similar Markov process with index $1$ (instead of $1/2$), but the results can be applied to $\sqrt{Y_t}$ under 
$P^\uparrow_x$, which is self-similar with index $1$ and such that the Laplace exponent of the associated L\'evy process is $\lambda\mapsto \psi(\lambda/2)$. The same 
remark applies to the forthcoming calculations. We finally note that the law of $(\sqrt{\lambda}\,Y_{\lambda t})_{t\geq 0}$ under $P^\uparrow_0$ is $P^\uparrow_{0}$.

\begin{proposition}
\label{peri-pro-ss}
The distribution of $(Z_t)_{t\geq 0}$ is $P_0^{\uparrow}$.
\end{proposition}
 \begin{proof}
 We claim that,  for every $t>0$ and $\lambda_1,\lambda_2\geq 0$ such that:
 \begin{equation}\label{eq:lambda:1:lambda:2:fubini}
 \Big(1+t\big(1+\lambda_1\big)^{1/2}\Big)^2 \lambda_2<1,
 \end{equation}
 we have:
 \begin{equation}\label{eq:claim:two:points}
 \mathbb{E}\Big[\exp\Big(-\frac{3}{2}\lambda_{1} Z_{1}-\frac{3}{2}\lambda_{2} Z_{1+t}\Big)\Big]=E_0^{\uparrow}\Big[\exp\Big(-\frac{3}{2}\lambda_{1} Y_{1}-\frac{3}{2}\lambda_{2} Y_{1+t}\Big)\Big].
 \end{equation} 
  The factors $3/2$ are present only to help simplifying the expressions in the calculations below.
  Let us explain why the proposition follows from \eqref{eq:claim:two:points}. First, using \eqref{eq:claim:two:points} and a scaling argument,
  we get that, for every $0<s<t$, the pair $(Z_s,Z_t)$ is distributed as $(Y_s,Y_t)$ under $P^\uparrow_0$. 
Therefore $Z$ and $Y$ (under $P^\uparrow_0$) have the same two-dimensional marginal distributions. 
  Since both processes are also Markovian (use Lemma \ref{eq:two:points} and the subsequent comments in the case of $Z$),
  the desired result follows.

It remains to prove \eqref{eq:claim:two:points}.  Observe that, by Lemma \ref{eq:two:points} and the known distribution of $Z_{1+t}$, we have
\begin{equation}
\label{formu-23}
 \mathbb{E}\Big[\exp\Big(-\frac{3}{2}\lambda_{1} Z_{1}-\frac{3}{2}\lambda_{2} Z_{1+t}\Big)\Big]= \left(1+\lambda_{1} +\lambda_{2}\Big(1+t\big(1+\lambda_{1} )^{\frac{1}{2}}\Big)^{2}\right)^{-1}, \end{equation}
 for every $t>0$. It remains to compute the right-side hand term of \eqref{eq:claim:two:points}. To this end, we use \cite[Proposition 1]{BY} (applied to the
 process $(\sqrt{Y_t})_{t\geq 0}$ under $\P^\uparrow_x$) to get, for every integer $p\geq 1$,
 $$ E_{x}^\uparrow\Big[Y_{t}^{p}\Big]=x^{p}+ \sum \limits_{k=1}^{2p}\frac{\prod_{\ell=0}^{k-1} \psi(p-\frac{\ell}{2})}{ k !} x^{p-\frac{k}{2}} t^k.$$
Next, a direct computation using the definition of $\psi$ gives 
  $$ E_{x}^\uparrow\Big[Y_{t}^{p}\Big]= \sum \limits_{k=0}^{2p} \big(\frac{2}{3}\big)^{\frac{k}{2}}  \binom{2p}{k} \frac{p !}{\Gamma(p+1-\frac{k}{2})}x^{p-\frac{k}{2}} t^k.$$
  Recall that the law of $Y_1$ under $P^\uparrow_0$ is exponential with mean $2/3$. By property (ii) again, we have
\begin{align*}
&\big(\frac{3}{2}\big)^{p}E_0^{\uparrow}\Big[\exp\big(-\frac{3}{2}\lambda_{1} Y_{1}) Y_{1+t}^{p} \Big]\\
&\quad=\big(\frac{3}{2}\big)^{p} E_0^\uparrow\Big[\exp\big(-\frac{3}{2}\lambda_{1} Y_{1}) E_{Y_1}^\uparrow\Big[Y_{t}^{p}\Big] \Big] \\
&\quad= \frac{3}{2}\sum \limits_{k=0}^{2p} \Bigg(\int_{0}^{\infty}\mathrm{d} x ~\exp\big(- \frac{3}{2}(\lambda_{1}+1)x\big)(\frac{3x}{2})^{p-\frac{k}{2}} \Bigg) \binom{2p}{k} \frac{p !}{\Gamma(p+1-\frac{k}{2})} t^k \\
&\quad=p!\, \sum \limits_{k=0}^{2p} \binom{2p}{k}  \frac{t^k}{(1+\lambda_1)^{p+1-\frac{k}{2}}} \\
&\quad= \frac{p!}{(1+\lambda_1)} \Big(\big(1+ \lambda_1\big)^{-1/2}+t\Big)^{2p}.
\end{align*}
It follows that the radius of convergence of the power series
$$z\mapsto \sum_{p=0}^\infty \big(\frac{3}{2}\big)^{p}E_0^{\uparrow}\Big[\exp\big(-\frac{3}{2}\lambda_{1} Y_{1})\, Y_{1+t}^{p} \Big] \; \frac{z^{p}}{p !} $$
is $(1+t\big(1+\lambda_1\big)^{1/2})^{-2}$. Hence, for every $\lambda_1,\lambda_2\geq 0$ satisfying \eqref{eq:lambda:1:lambda:2:fubini}, we get
\begin{align*}
E_0^{\uparrow}\Big[\exp\Big(-\frac{3}{2}\lambda_{1} Y_{1}-\frac{3}{2}\lambda_{2} Y_{1+t}\Big)\Big]&=\frac{1}{(1+\lambda_1)} \sum_{p=0}^\infty \Big(\big(1+ \lambda_1\big)^{-1/2}+t\Big)^{2p}\, \big(-\lambda_2\big)^p\\
&= \left(1+\lambda_{1} +\lambda_{2}\Big(1+t\big(1+\lambda_{1} )^{\frac{1}{2}}\Big)^{2}\right)^{-1},
\end{align*}
and, by comparing with \eqref{formu-23}, we get the desired formula \eqref{eq:claim:two:points}.
\end{proof}

\section{A new construction of the Brownian half-plane}
\label{sec:new}

In this section, we will give a new construction of the Brownian half-plane, which is an analog
of a result proved in \cite{LGR3} for the Brownian disk. This construction will be useful
in Section \ref{sec:segment} when we study the complement of the hull centered
on a boundary segment. 

\subsection{A Brownian disk with a random perimeter}
\label{sec:Bdisk}

The goal of this section is to recall the construction of a Brownian disk with a random perimeter that
was given in \cite{LGR3}, to which we refer for more details. 
We consider a normalized Brownian excursion $(\be_t)_{0\leq t\leq 1}$,
and, conditionally on $(\be_t)_{0\leq t\leq 1}$, a Poisson point measure $\n=\sum_{j\in J}\delta_{(t_j,\omega^j)}$ on $[0,1]\times \mathcal{S}$
with intensity
$$2\,\dd t\,\N_{\sqrt{3}\,\be_t}(\dd \omega).$$
For every $j\in J$, we consider the truncation  $\wt \omega^j:=\mathrm{tr}_0(\omega^j)$ of $\omega^j$ at level $0$, and we 
 let $\mathfrak{T}^\star$ be
the compact metric space obtained from
the disjoint union
\begin{equation}
\label{tree-disk}
[0,1] \cup \Big(\bigcup_{j\in J} \t_{(\tilde\omega^j)}\Big)
\end{equation}
by identifying the root $\rho_{(\tilde\omega^j)}$ of $\t_{(\tilde\omega^j)}$
with the point $t_j$ of $[0,1]$, for every $j\in J$. The metric $\dd_{\mathfrak{T}^\star}$ on $\mathfrak{T}^\star$ is defined in the obvious manner, so that
the restriction of $\dd_{\mathfrak{T}^\star}$ to each tree $\t_{(\tilde\omega^i)}$ is 
the metric $d_{(\tilde\omega^i)}$, and $\dd_{\mathfrak{T}^\star}(u,v)=|v-u|$ if $u,v\in[0,1]$ (compare with the Caraceni-Curien construction of Section \ref{car-cur}, and see \cite[Section  4]{LGR3} for more details). 
The volume measure on $\mathfrak{T}^\star$ is the sum of the volume measures on the trees $\t_{(\tilde\omega^j)}$, $j\in J$.
We then assign labels $(\ell^\star_a)_{a\in\mathfrak{T}^\star}$ to the points of $\mathfrak{T}^\star$.
 If $a=s\in[0,1]$, we
take $\ell^\star_a:=\sqrt{3}\,\be_s$, and if $a\in \t_{(\tilde\omega^j)}$ for some $j\in J$, we 
simply let $\ell^\star_a$ be the label of $a$ in $ \t_{(\tilde\omega^j)}$. We note that the function $a\mapsto\ell^\star_a$
is continuous on $\mathfrak{T}^\star$, and that
labels $\ell^\star_a$ are
nonnegative for every $a\in\mathfrak{T}^\star$ (because we replaced each $\omega^j$ by its truncation $\wt\omega^j$). We 
define the boundary of $\mathfrak{T}^\star$ by $\partial\mathfrak{T}^\star:=\{a\in\mathfrak{T}^\star: \ell^\star_a=0\}$ and note that $0,1\in\partial\mathfrak{T}^\star$.

If $\Sigma^\star:=\sum_{j\in J}\sigma(\wt\omega^j)$ is the total mass of the volume measure, we can define a clockwise exploration $(\ee^\star_t)_{0\leq t\leq\Sigma^\star}$
of $\mathfrak{T}^\star$ by concatenating the mappings $p_{(\tilde\omega^j)}:[0,\sigma(\wt\omega^j)]\la \t_{(\tilde\omega^j)}$ in the 
order prescribed by the $t_j$'s (again see \cite{LGR3} for more details). Note that $\ee^\star_0=0$ and $\ee^\star_{\Sigma^\star}=1$.

Similarly as in Section \ref{car-cur}, the clockwise exploration allows us to define ``intervals''
in $\mathfrak{T}^\star$. For
$s,t\in[0,\Sigma^\star]$,  if $s>t$, we set $[s,t]^\star:=[s,\Sigma^\star]\cup [0,t]$ and if
$s\leq t$, $[s,t]^\star:=[s,t]$ is the usual interval.  
Then, for every $u,v\in\mathfrak{T}^\star$, there is a smallest interval $[s,t]^\star$, with $s,t\in[0,\Sigma^\star]$, such that
$\ee^\star_s=u$ and $\ee^\star_t=v$, and we define  
$$[u,v]^\star:=\big\{\ee^\star_r:r\in[s,t]^\star\big\}.$$

We then set, for every $a,b\in\mathfrak{T}^\star\backslash\partial\mathfrak{T}^\star$,
\begin{equation}
\label{Dbar1}
D_\star^\circ(a,b):=\ell^\star_a+\ell^\star_b - 2\max\Big(\min_{c\in[a,b]^\star}\ell^\star_c,\min_{c\in[b,a]^\star}\ell^\star_c\Big)
\end{equation}
if the maximum in the right-hand side is positive, and $D_\star^\circ(a,b):=\infty$ otherwise. Finally, we set,
for every $a,b\in\mathfrak{T}^\star\backslash\partial\mathfrak{T}^\star$, 
\begin{equation}
\label{Dbar2}
D_\star(a,b):=\inf_{a_0=a,a_1,\ldots,a_{p-1},a_p=b}\sum_{k=1}^p D_\star^\circ(a_{k-1},a_k)
\end{equation}
where the infimum is over all choices of the integer $p\geq 1$
and of the points $a_1,\ldots,a_{p-1}$ in $\mathfrak{T}^\star\backslash\partial\mathfrak{T}^\star$. It is not hard to
verify that $D_\star(a,b)<\infty$ (see Proposition 30 (i) in \cite{Disks} for a very similar
argument). The mapping $(a,b)\mapsto D_\star(a,b)$
is continuous on $(\mathfrak{T}^\star\backslash\partial\mathfrak{T}^\star)\times (\mathfrak{T}^\star\backslash\partial\mathfrak{T}^\star)$, 
and has a unique continuous extension to $\mathfrak{T}^\star\times \mathfrak{T}^\star$, which is a pseudo-metric on $\mathfrak{T}^\star$
\cite[Proposition 5]{LGR3}. Moreover, by \eqref{Dbar1} and \eqref{Dbar2}, we have 
$D_\star(a,b)\geq |\ell^\star_a-\ell^\star_b|$. 

 We then consider the quotient space $\U:=\mathfrak{T}^\star/\{D_\star=0\}$, and the canonical projection $\Pi_\star:\mathfrak{T}^\star\la \U$. The function  $(a,b)\mapsto D_\star(a,b)$ induces 
a metric on $\U$, which we still denote by $D_\star$, and the metric space 
$(\U,D_\star)$ is equipped with the pushforward of the 
volume measure on $\mathfrak{T}^\star$ under $\Pi_\star$, which is 
denoted by $\mathbf{V}_\star$. We also write $\partial_0\U=\Pi_\star([0,1])$ and
$\partial_1\U=\Pi_\star(\partial\mathfrak{T}^\star)$. Finally, we note that the equivalence class
of $0$ (or that of $1$) in the quotient space $\U:=\mathfrak{T}^\star/\{D_\star=0\}$ is a singleton
(this follows from \cite[Proposition 5]{LGR3}, which describes the pairs $(a,b)$ in $\mathfrak{T}^\star$
such that $D^\star(a,b)=0$).

\begin{theorem} {\rm\cite[Theorem 16]{LGR3}}
\label{cons-disk}
The random measure metric space $(\U,D_\star, \mathbf{V}_\star,\Pi_\star(0))$ is a free Brownian disk with a random 
boundary size distributed according to the measure $\frac{3}{2}\,\mathbf{1}_{\{z>1\}}z^{-5/2}\,\dd z$, which is pointed at
a uniform boundary point. Furthermore, the boundary $\partial\U$ is equal to $\partial_0\U\cup \partial_1\U$. 
\end{theorem}

In contrast with Section \ref{Spa-Mar}, we view here the free Brownian disk as a (random) pointed measure metric space: If we condition
the boundary size of $\U$ to be equal to $S>0$, the space $(\U,D_\star, \mathbf{V}_\star,\Pi_\star(0))$  has the same
distribution as the space $(\D_{(S)},D_{(S)},V_{(S)},\Lambda_{(S)}(0))$, with the notation introduced 
at the beginning of the proof of Theorem \ref{peeling-HP}. 

We note that labels $\ell^\star_x$ make sense for $x\in\U$ (because 
$D_\star(a,b)=0$ implies $\ell^\star_a=\ell^\star_b$) and furthermore we have $D_\star(x,\partial_1\U)=\ell^\star_x$
for every $x\in\U$ (see \cite[Section 4.2]{LGR3}).

\subsection{Constructing the Brownian half-plane}

Let us start from a three-dimensional Bessel process $(R_t)_{t\geq 0}$
with $R_0=0$  and then consider a random point measure 
$\mathcal{N}_\infty=\sum_{j\in J_\infty}\delta_{(t_j,\omega^j)}$  on $\mathbb{R}_+\times\mathcal{S}$, such that, conditionally on $(R_t)_{t\geq 0}$, the measure $\mathcal{N}_\infty$ is Poisson with intensity:
$$2\,\mathrm{d}t ~\mathbb{N}_{\sqrt{3}R_t}(\rm{d} \omega).$$
For every $j\in J_\infty$, we let $\wt\omega^j$ be the truncation  of $\omega^j$ at level $0$. 

In a way similar to Sections \ref{car-cur} and \ref{sec:Bdisk}, we then introduce the geodesic space
$\mathfrak{T}_\infty$ which is obtained from the disjoint union
$$[0,\infty) \cup \Bigg(\bigcup_{j\in J_\infty} \t_{(\tilde\omega^j)}\Bigg)$$
by identifying the root of $\t_{(\tilde\omega^j)}$ with the point $t_j$ of $[0,\infty)$, for every $j\in J_\infty$.  We interpret $[0,\infty)$ as the ``spine'' of 
$\mathfrak{T}_\infty$. We also define the volume measure on $\mathfrak{T}_\infty$ as the sum of the volume measures on the trees $\t_{(\tilde\omega^j)}$, $j\in J_\infty$.

We next assign labels to $\mathfrak{T}_\infty$ by taking $\ell^\infty_u:=\sqrt{3}R_u$, if $u\in \mathbb{R}_+$, and 
by letting $\ell^\infty_u$ be the label of $u$ in $\t_{(\tilde\omega^j)}$  if $u\in \t_{(\tilde\omega^j)}$. We can also introduce a clockwise exploration
$(\mathcal{E}_{t}^\infty)_{t\geq 0}$ of  $\mathfrak{T}_\infty$, defined by concatenating the mappings $p_{(\tilde \omega^j)}$ in the order prescribed by the $t_j$'s. 
As in to Sections \ref{car-cur} and \ref{sec:Bdisk}, we then define intervals on $\mathfrak{T}_\infty$. For $s,t\in\R_+$, we set $[s,t]^\prime:= [s,\infty)\cup[0,t]$ if $s>t$ and $[s,t]^\prime:= [s,t]$ if $s\leq t$.
Then, for every $u, v \in \mathfrak{T}_\infty$, we set $[u,v]^\prime_\infty:=\{\ee_r^\infty:~r\in[s,t]^\prime\}$, where $s,t\in\R_+$ are such  that
$\ee_s^\infty=u$ and $\ee_t^\infty=v$ and the interval $[s,t]^\prime$ is as small as possible. 

Let $\partial \mathfrak{T}_\infty:= \{u\in \mathfrak{T}_\infty:~\ell_u^\infty=0\}$. For every $a,b\in\mathfrak{T}_\infty\backslash\partial\mathfrak{T}_\infty$, we set
\begin{equation}
\label{D_infty-1}
D_\infty^{\circ}(a,b):=\ell_a^\infty+\ell_b^\infty - 2\max\Big(\min_{c\in[a,b]^\prime_\infty}\ell_c^\infty,\min_{c\in[b,a]^\prime_\infty}\ell_c^\infty\Big)
\end{equation}
if the maximum in the right-hand side is positive, and $D_\infty^\circ(a,b):=\infty$ otherwise. 
Finally, in exactly the same way as we defined $D_\star$ from $D^\circ_\star$, we set,
for every $a,b\in\mathfrak{T}_\infty\backslash\partial\mathfrak{T}_\infty$, 
\begin{equation}
\label{D_infty-2}
D_\infty(a,b):=\inf_{a_0=a,a_1,\ldots,a_{p-1},a_p=b}\sum_{k=1}^p D_\infty^{\circ}(a_{k-1},a_k)
\end{equation}
where the infimum is over all choices of the integer $p\geq 1$
and of the points $a_1,\ldots,a_{p-1}$ in $\mathfrak{T}_\infty\backslash\partial\mathfrak{T}_\infty$. By
arguments similar to the proof of \cite[Proposition 30]{Disks}, one verifies that the mapping $(a,b)\mapsto D_\infty(a,b)$ takes finite values and
is continuous on $(\mathfrak{T}_\infty\backslash\partial\mathfrak{T}_\infty)\times (\mathfrak{T}_\infty\backslash\partial\mathfrak{T}_\infty)$, and that we have 
$D_\infty(a,b)\geq |\ell_a^\infty-\ell_b^\infty|$.

\begin{proposition}\label{extension-D_infty}
The function $(a,b)\mapsto D_\infty(a,b)$ has a continuous extension to $\mathfrak{T}_\infty\times \mathfrak{T}_\infty$, which is
a pseudo-metric on $\mathfrak{T}_\infty$. Furthermore,
the property $D_\infty(a,b)=0$ holds if and only if either $a$ and $b$ both belong to
$\mathfrak{T}_\infty\backslash\partial\mathfrak{T}_\infty$ and $D^{\circ}_\infty(a,b)=0$, or $a$ and $b$
both belong to $\partial\mathfrak{T}_\infty$ and we have $\{a,b\}=\{\ee^\infty_s,\ee^\infty_t\}$,
for some $0\leq s\leq t< \infty$ such that $\ell_{\ee^\infty_r}^\infty>0$ for every $r\in(s,t)$.
\end{proposition}

This is an analog of Proposition 5 in \cite{LGR3}, which deals with the metric space $(\mathfrak{T}^\star,D_\star)$ introduced above. The proof
is exactly the same and we omit the details.

We  then consider the quotient space $\mathfrak{H}'=\mathfrak{T}_\infty/\{D_\infty=0\}$, and the canonical projection $\Pi_\infty:\mathfrak{T}_\infty\mapsto\mathfrak{H}'$.  The metric space 
$(\mathfrak{H}',D_\infty)$ is equipped with the distinguished point $\Pi_\infty(0)$ and the pushforward of the 
volume measure on $\mathfrak{T}_\infty$ under $\Pi_\infty$,  which will be 
denoted by $\mathbf{V}_\infty$. We also write $\partial_\circ\mathfrak{H}'=\Pi_\infty([0,\infty))$ and
$\partial_1\mathfrak{H}'=\Pi_\infty(\partial\mathfrak{T}_\infty)$.  Then we notice that  labels $\ell_x^\infty$ make sense for $x\in\mathfrak{H}'$ (again since 
$D_\infty(a,b)=0$ implies $\ell_a^\infty=\ell_b^\infty$). 

We finally define a decorating curve $\Lambda'$ of $\HH'$. First, we take $\Lambda'(t)=\Pi_\infty(t)$
for every $t\geq 0$. To define $\Lambda'(t)$ when $t\leq 0$,  we set, for every $s\geq 0$,
$$L^{\mathfrak{H}',0}_s:=\sum_{j\in J_\infty} L^0_{(s-\alpha_j)^+}(\wt \omega^j),$$
where $\alpha_j:=\inf\{s\geq 0:\mathcal{E}^\infty_s\in\t_{(\tilde \omega^j)}\}$, and we recall that
$(L^0_s(\wt\omega^j))_{s\geq 0}$ denotes the exit local time at $0$ of
the truncated snake trajectory $\wt\omega^j$ (see the end of Section \ref{sna-mea}). Then, for every $t\leq 0$, if
$$\tau_t:=\inf\{s\geq 0:L^{\mathfrak{H}',0}_s\geq -t\},$$
we define $\Lambda'(t):=\Pi_\infty(\mathcal{E}^\infty_{\tau_t})$.
Using arguments similar to those used to study the path $\Lambda^{\bullet,r}$ at the end of Section \ref{car-cur},
one verifies that the path $t\mapsto\Lambda(t)$ is continuous and injective, and
 we have $\partial_0\mathfrak{H}'=\{\Lambda'(t):t\geq 0\}$ and $\partial_1\mathfrak{H}'=\{\Lambda'(t):t\leq 0\}$ .

\begin{theorem}
\label{new-const-half-plane}
The random curve-decorated measure metric space $(\mathfrak{H}',D_\infty,\mathbf{V}_\infty,\Lambda')$ is a curve-decorated Brownian half-plane. Furthermore, we have $D_\infty(x,\partial_1\mathfrak{H}')=\ell_x^\infty$,
for every $x\in\mathfrak{H}'$, and in particular  the process $(D_\infty(\Lambda'(t),\partial_1\HH'))_{t\geq 0}$ is distributed as a three-dimensional Bessel process
started from $0$.
\end{theorem}

The best way to understand the relation between Theorem \ref{new-const-half-plane} and the Caraceni construction
of Section~\ref{car-cur} is to consider the analogous results for the Brownian disk, namely Theorem \ref{cons-disk} above
and the construction of the Brownian disk ``viewed from a boundary point'' in \cite{Repre}. The construction
of Theorem \ref{cons-disk} yields a Brownian disk $\U$ with a random boundary size, where labels correspond to distances
from the part $\partial_1\U$ of the boundary, and these distances evolve like a Brownian excursion $\be$
along the complementary part $\partial_0\U$ of the boundary. Conditioning $\partial_1\U$ to be a single point (equivalently, none
of the snake trajectories $\omega^j$ in Section \ref{sec:Bdisk} hits $0$, so that the boundary size of $\U$ is $1$) turns the Brownian excursion $\be$ 
into a five-dimensional Bessel bridge, and we recover the construction of \cite{Repre}
where labels corresponds to distances from a distinguished point of the boundary
(in particular these distances evolve like a five-dimensional Bessel bridge along the boundary). The latter conditioning is
degenerate, but the preceding assertions can be made rigorous and explain the relation between Theorem \ref{cons-disk} and \cite{Repre}.
Then, at least informally, the relation between the two half-plane constructions can be obtained by
letting the boundary sizes of the Brownian disks tend to infinity: in this limit, the Brownian excursion $\be$
behaves locally like a three-dimensional Bessel process, and the five-dimensional Bessel bridge near its initial and terminal times
gives rise to two independent five-dimensional Bessel processes.

Before turning to the proof of Theorem \ref{new-const-half-plane}, we state a preliminary lemma. For every $\ve>0$, we let
  $\mathfrak{T}^{\ve}_{\infty}$ be the closed subset of $\mathfrak{T}_\infty$ consisting of the
 part $[0,\ve]$ of the spine and of the subtrees branching off $[0,\ve]$.

\begin{lemma}\label{inf:T:esp}
For every $\ve>0$, the following properties hold a.s.

\smallskip
\noindent{\rm(i)} Labels vanish both on $\mathfrak{T}^{\ve}_\infty\backslash\{0\}$ and on $\mathfrak{T}_\infty\setminus\mathfrak{T}^\ve_\infty$.

\smallskip
\noindent{\rm(ii)} We have
$$\inf \limits_{v\in\mathfrak{T}_\infty\setminus \mathfrak{T}^{\varepsilon}_\infty } D_\infty(0, v)>0.$$
\end{lemma}

\proof \noindent{\rm(i)} To prove that labels vanish on $\mathfrak{T}_\infty\setminus\mathfrak{T}^\ve_\infty$, it is enough to verify that there exists $j\in J_\infty$ with $t_j>\ve$, such that $W_*(\omega^j)\leq 0$. Recall that, conditionally on $(R_t)_{t\geq 0}$, the measure $\mathcal{N}_\infty$ is Poisson with intensity $2\mathbf{1}_{\{t\geq 0\}}\,\mathrm{d}t \,\mathbb{N}_{\sqrt{3}R_t}(\rm{d} \omega).$  Consequently, an application of \eqref{hittingpro} gives
$$\mathbb{P}\Big(W_*(\omega^j)>0\hbox{ for every }j\in J_\infty\hbox{ such that } t_j>\ve\Big)=\mathbb{E}\Big[\exp\big(-\int_\ve^\infty \frac{\dd t}{R_t^2}\big)\Big]. $$
The fact that $\int_\ve^\infty \dd t~R_t^{-2}=\infty$, a.s., yields the desired result. The same argument applies to verify that 
labels vanish on $\mathfrak{T}^{\ve}_\infty\backslash\{0\}$.

\noindent{\rm(ii)} 
Let $v_{(\varepsilon)}$ be the last point of $\mathfrak{T}^{\ve}_{\infty}\cap\partial \mathfrak{T}_\infty$ visited by
 the exploration
$(\mathcal{E}_{t}^\infty)_{t\geq 0}$ of  $\mathfrak{T}_\infty$ and let $r_{(\ve)}\in(0,\infty)$ such that $\mathcal{E}_{r_{(\ve)}}^\infty=v_{(\ve)}$.
We then claim that, for any $v\in \mathfrak{T}_\infty\backslash\mathfrak{T}^{\varepsilon}_\infty$, 
\begin{equation}
\label{claim}D_\infty(0,v)\geq \inf\limits_{u\in \llbracket v_{(\varepsilon)},\infty \llbracket}  D_\infty(0,u),
\end{equation}
 where  $\llbracket v_{(\varepsilon)},\infty\llbracket$ stands for the geodesic line connecting $v_{(\varepsilon)}$ to $\infty$ in  $\mathfrak{T}_\infty$. 
 Let us justify our claim. The continuity of
 $v\mapsto D_\infty(0,v)$ allows us to assume that $v\notin \partial\mathfrak{T}_\infty$. Then, let $\delta\in(0,\ve)$. We observe that, in formula \eqref{D_infty-2}
 applied to $D_\infty(\delta,v)$, we may restrict our attention to points $a_0,a_1,\ldots,a_p$ such that there is (at least) one value of $j\in\{1,\ldots,p-1\}$
 such that $a_j\in \llbracket v_{(\varepsilon)},\infty\llbracket$: if not the case, by letting $k$ be the first index $j\in\{1,\ldots,p\}$ such 
 that $a_j\in \{\mathcal{E}^\infty_t:t>r_{(\ve)}\}$, 
 we would have $v_{(\ve)}\in [a_{k-1},a_k]_\infty^\prime$, and thus $D^\circ_\infty(a_{k-1},a_{k})=\infty$.
It follows that $D_\infty(\delta,v)$ is bounded below by
 the right-hand side of \eqref{claim}, and our claim follows by letting $\delta\to 0$.

The proof then reduces to checking that the right-hand side of \eqref{claim} is positive. We argue by contradiction. Asume that there is a sequence $(u_n)_{n\geq 1}$ in  
 $\llbracket v_{(\varepsilon)},\infty\llbracket$ such that $D_\infty(0, u_n)\la 0$ as $n\to\infty$. The bound $D_\infty(a,b)\geq |\ell_a^\infty-\ell_b^\infty|$ ensures that $\ell_{u_n}^\infty\la 0$ and this implies that $u_n\la  v_{(\varepsilon)}$ in $\mathfrak{T}_\infty$.  It follows that $D_\infty(0, v_{(\varepsilon)})=0$, which contradicts Proposition \ref{extension-D_infty}. 
\endproof 

\proof[Proof of Theorem \ref{new-const-half-plane}] 
For the sake of simplicity, we will not consider the decorating curve and we will content ourselves with proving that the random pointed 
measure metric space $(\mathfrak{H}',D_\infty,\mathbf{V}_\infty,\Pi_\infty(0))$ is a Brownian half-plane, whose
boundary is $\partial_0\HH'\cup\partial_1\HH'$. 
With a little more work, the arguments that follow can be extended to also include the decorating curve (instead of Theorem \ref{cons-disk} above, we
need the precise form of \cite[Theorem 16]{LGR3} including the decorating curve of $(\U,D_\star, \mathbf{V}_\star)$, which is in fact defined
in a way very similar to $\Lambda'$).

Recall the construction of the Brownian disk
$(\U,D_\star,\mathbf{V}_\star,\Pi_\star(0))$ in the previous section. The scaling property of the Brownian disk implies that, for every $\lambda>0$, the  random pointed measure metric space  $\lambda\cdot \U:=(\U,\lambda D_\star,\lambda^{4} \mathbf{V},\Pi_\star(0))$ is a 
 free Brownian disk 
with a random perimeter distributed according to the measure $\frac{3}{2}\,\lambda^{3/2}\,z^{-5/2}\,\mathbf{1}_{\{z>\lambda\}}\,\dd z$, which is pointed at a uniform boundary point.  We will prove that
\begin{equation}\label{local:limit}
\lambda \cdot  \U\build{\la}_{\lambda\to\infty}^{\rm(d)}\mathfrak{H}'
\end{equation}
in distribution in the sense of the space $\M^{\mathrm{GHPU}}_\infty$ (recall that our
pointed measure metric spaces are viewed as elements of $\M^{\mathrm{GHPU}}_\infty$ whose decorating curve is constant). The 
fact that  $(\mathfrak{H}',D_\infty,\mathbf{V}_\infty,\Pi_\infty(0))$ is a Brownian half-plane will follow  since one knows that the Brownian half-plane is the limit 
in distribution (in the space $\M^{\mathrm{GHPU}}_\infty$) of 
Brownian disks pointed at a uniform boundary point whose boundary size tends to $\infty$ --- this follows from the coupling argument 
already used at the beginning of the proof of Theorem \ref{peeling-HP}.

The proof of \eqref{local:limit} is based again on a coupling argument. 
Let $K>0$ and $\delta>0$. Our claim \eqref{local:limit} will follow if we can prove that, for $\lambda$ large enough, we can couple $\U$ and $\mathfrak{H}'$ in such a way that
 the balls $B_K(\lambda\cdot \U)$ and $B_K(\mathfrak{H}')$ are isometric with probability at least $1-\delta$
 (we require that the isometry preserves the volume measure and the distinguished point).  Equivalently, using a scaling argument, it suffices to prove that, for $\eta>0$ small enough, 
  $\U$ and $\mathfrak{H}'$ can be coupled so that 
$B_{\eta}(\U)$ and $B_{\eta}(\mathfrak{H}')$ are isometric with probability at least $1-\delta$
(again with an isometry preserving the volume measure and the distinguished point). 

Recall the point measure $\mathcal{N}=\sum_{j\in J} \delta_{(t_i,\omega^i)}$ introduced at the beginning of Section \ref{sec:Bdisk} and used to construct $\mathfrak{T}^\star$, 
and the point measure $\mathcal{N}_\infty=\sum_{j\in J_\infty} \delta_{(t_i,\omega^i)}$ used to construct $\mathfrak{T}_\infty$. For every $\ve >0$, set 
$$\n^\varepsilon= \sum \limits_{j\in J,t_j\leq \ve} \delta_{(t_j,\omega^j)},\quad \mathcal{N}^\ve_\infty=\sum_{j\in J_\infty,t_j\leq\ve}\delta_{(t_j,\omega^j)}.$$
We now fix $\delta>0$ and claim that:
\begin{itemize}
\item[$1.$] For $\ve\in(0,1)$ small enough,  $(\be,\n)$ and $(R,\n_\infty)$ can be coupled in such a way that
the equality $((\be_t)_{t\leq\varepsilon},\n^{\varepsilon})= ((R_t)_{t\leq\varepsilon},\n_\infty^\varepsilon)$
holds with probability at least $1-\frac{\delta}{2}$. 
\item[$2.$] For $\ve\in(0,1)$ small enough, we can choose $\eta_0>0$ so that for every $0<\eta\leq \eta_0$, we have
$$B_\eta(\U)=B_\eta(\mathfrak{H}')$$
on the event where $((\be_t)_{t\leq\varepsilon},\n^{\varepsilon})= ((R_t)_{t\leq\varepsilon},\n_\infty^\varepsilon)$, except possibly on an event of probability at most $\frac{\delta}{2}$. 
\end{itemize} 
The formula $B_{\eta}(\U)=B_{\eta}(\mathfrak{H}')$ in Property 2 is understood as an equality of pointed measure 
metric spaces modulo isometries ($B_{\eta}(\U)$, resp.~$B_{\eta}(\mathfrak{H}')$,
is equipped with the distinguished point and the restriction of the volume  measure of $\U$, resp.~of $\mathfrak{H}'$).

 As explained above, our claim \eqref{local:limit} follows from Properties 1 and 2. Property 1 is a consequence of the following classical  fact. For every $\delta$, we can find $\ve\in(0,1)$  and a coupling between  $\be$ and $R$ such that the equality $(\be_t)_{0\leq t\leq\varepsilon}= (R_t)_{0\leq t\leq\varepsilon}$
holds with probability at least $1-\frac{\delta}{2}$ (see e.g.~Proposition 3 in \cite{plane} and the proof of Proposition 4 in the same reference for a stronger statement).

 It remains to verify Property 2. Informally, Property 1 provides a coupling under which the pairs $(\be,\mathcal{N})$ and 
$(R, \mathcal{N}_\infty)$ coincide ``near the origin''. Since the pair $(\be,\mathcal{N})$ encodes the Brownian disk $\U$,
and the pair $(R, \mathcal{N}_\infty)$ encodes the half-plane $\mathfrak{H}'$, it is not surprising that this coupling 
yields the equality of the balls $B_\eta(\U)$ and $B_\eta(\mathfrak{H}')$ for $\eta$ small enough. Making this argument 
rigorous is however not immediate, in particular because the distances in $B_\eta(\U)$ or in $B_\eta(\mathfrak{H}')$
may depend on the encoding processes ``far from the origin''. For this reason, we will provide some details.

 By Property 1, we can fix $\ve\in(0,1)$ small and assume that
the event where $((\be_t)_{t\leq\varepsilon},\n^{\varepsilon})= ((R_t)_{t\leq\varepsilon},\n_\infty^\varepsilon)$  has probability greater than $1-\frac{\delta}{2}$. We  
denote the latter event by $\mathcal{A}_1$.

By Lemma \ref{inf:T:esp}, we can fix $\eta>0$ small enough so that the set $\{v\in\mathfrak{T}_\infty: D_\infty(0,v)\leq 4\eta\}$ is contained 
in $\mathfrak{T}_\infty^\ve$, except possibly on an event of probability at most $\frac{\delta}{6}$. Moreover, if the latter property holds, we have also
\begin{equation}
\label{tech-dist}
D_\infty(u,v) = \build{\inf_{u_0=u,u_1,\ldots,u_p=v}}_{u_1,\ldots,u_{p-1}\in \mathfrak{T}_\infty^\ve\setminus \partial \mathfrak{T}_\infty}^{} \sum_{i=1}^pD^{\circ}_\infty(u_{i-1},u_i),
\end{equation}
for every $u,v\in\mathfrak{T}^{\ve}_\infty\setminus \partial \mathfrak{T}_\infty$
such that $D_\infty(0,u)\leq \eta$ and $D_\infty(0,v)\leq \eta$. Let us explain why  \eqref{tech-dist} holds. 
Suppose that $u,v\in\mathfrak{T}^{\ve}_\infty\setminus \partial \mathfrak{T}_\infty$ are such that $D_\infty(0,u)\leq \eta$, $D_\infty(0,v)\leq \eta$, so that
in particular $D_\infty(u,v)\leq 2\eta$. We can find $u_0=u,u_1,\ldots,u_q=v$
in $\mathfrak{T}_\infty\setminus   \partial\mathfrak{T}_\infty$ such that 
$$\sum_{i=1}^q D_\infty^{\circ}(u_{i-1},u_i)<D_\infty(u,v)+\eta\leq 3\eta.$$
The triangle inequality then implies that $D_\infty(u,u_i)< 3\eta$ and $D_\infty(0,u_i)<4\eta$, for every $i\in\{0,1,\ldots,q\}$.  Since we assumed
that $\{v\in\mathfrak{T}_\infty: D_\infty(0,v)\leq 4\eta\}$ is contained 
in $\mathfrak{T}_\infty^\ve$, it follows that $u_i\in  \mathfrak{T}^{\ve}_{\infty}$, for every $i\in\{0,1,\ldots,q\}$. 
In other words, in formula \eqref{D_infty-2} defining $D_\infty(u,v)$, we may restrict the infimum to the case where
all $u_i$'s belong to  $\mathfrak{T}^{\ve}_{\infty}$. This gives our claim \eqref{tech-dist}. Furthermore, when applying formula \eqref{D_infty-1} to compute the quantities $D^{\circ}_\infty(u_{i-1},u_i)$ in the right-hand side 
of \eqref{tech-dist}, we can restrict our attention to the case when the interval $[u_{i-1}, u_i]^\prime_\infty$ (resp.~$[u_i, u_{i-1}]^\prime_\infty$) is 
contained in $\mathfrak{T}^{\ve}_{\infty}$, since otherwise this interval contains 
the complement of $\mathfrak{T}^{\ve}_{\infty}$ and then the infimum of labels on $[u_{i-1}, u_i]^\prime_\infty$ is $0$ by Lemma \ref{inf:T:esp} (i).

Let us now discuss $D_\star(u,v)$ when $u,v\in B_\eta(\U)$. We write $\mathfrak{T}^{\star,\ve}$  for the closed subset of $\mathfrak{T}^\star$ 
consisting of the part $[0,\ve]$ of the ``spine'' $[0,1]$ and the subtrees branching off $[0,\ve]$. Then, we have 
$$\inf \limits_{v\in\mathfrak{T}^\star\setminus \mathfrak{T}^{\star,\varepsilon} } D_\star(0, v)>0,\quad \text{a.s.}~,$$
since we know that the equivalence class of $0$ in $\U$ is a singleton. Hence, for $\eta>0$ small, we get that the event where
$\{v\in \mathfrak{T}^\star:D_\star(0,v)\leq 4\eta\}\subset \mathfrak{T}^{\star,\ve}$ has probability at least
$1-\frac{\delta}{6}$. On the latter event, the same argument as for \eqref{tech-dist} then shows that, for every
$u,v\in\mathfrak{T}^{\star,\ve}\setminus \partial \mathfrak{T}^\star$ such that $D_\star(0,u)\leq\eta$ and 
$D_\star(0,v)\leq\eta$, we have 
\begin{equation}
\label{tech-dist2}
D_\star(u,v) = \build{\inf_{u_0=u,u_1,\ldots,u_p=v}}_{u_1,\ldots,u_{p-1}\in \mathfrak{T}^{\star,\ve}\setminus \partial \mathfrak{T}^\star}^{} \sum_{i=1}^pD^{\circ}_\star(u_{i-1},u_i),
\end{equation}
and moreover, when applying formula \eqref{Dbar1} to compute $D^{\circ}_\star(u_{i-1},u_i)$ we may discard the case when the interval $[u_{i-1},u_i]_\star$ (resp. $[u_{i},u_{i-1}]_\star$)
is not contained in $\mathfrak{T}^{\star,\ve}$. 

On the event $\mathcal{A}_1$, the labeled tree $\mathfrak{T}^{\star,\ve}$ is identified with $\mathfrak{T}^{\ve}_\infty$. Moreover, fixing $\eta>0$ small enough and
discarding an 
event $\mathcal{A}_2$ of probability at most $\frac{\delta}{3}$, we deduce from formulas \eqref{tech-dist} and \eqref{tech-dist2}, that  we have $D_\star(u,v) =D_\infty(u,v)$ whenever $u,v\in \mathfrak{T}^{\star,\ve}=\mathfrak{T}^{\ve}_\infty$ are such that
$D_\star(0,u)\vee D_\star(0,v)\leq\eta$ (which is equivalent to $D_\infty(0,u)\vee D_\infty(0,v)\leq\eta$ by \eqref{tech-dist} and \eqref{tech-dist2}). 
On the event $\mathcal{A}_1\backslash \mathcal{A}_2$, we also know that $\{v\in\mathfrak{T}_\infty: D_\infty(0,v)\leq 4\eta\}\subset\mathfrak{T}_\infty^\ve$
and $\{v\in \mathfrak{T}^\star:D_\star(0,v)\leq 4\eta\}\subset \mathfrak{T}^{\star,\ve}$. It then follows from this
discussion that, still on the event $\mathcal{A}_1\backslash \mathcal{A}_2$,  the identification of $\mathfrak{T}^{\star,\ve}$ with $\mathfrak{T}^{\ve}_\infty$
induces an isometry from $B_\eta(\U)$ onto $B_\eta(\mathfrak{H}')$, which clearly preserves the volume measures and the distinguished points. This completes the proof
of Property 2 and of the first assertion of the theorem. 

Let us finally discuss the boundary $\partial \mathfrak{H}'$ of $\mathfrak{H}'$. Recall that $\partial \mathfrak{H}'$ is defined as the set of all
points of $\mathfrak{H}'$ that have no neighborhood homeomorphic to the open unit disk.
The preceding identification of $B_\eta(\U)$ with $B_\eta(\mathfrak{H}')$ (except on
an event of probability at most $\delta$) also shows that, for every $\vartheta\in(0,\eta)$, the set $\partial \mathfrak{H}'\cap B_\vartheta(\mathfrak{H}')$ is identified with
$\partial\U\cap B_\vartheta(\U)$, which we know to be equal to $(\partial_0\U\cup\partial_1\U)\cap B_\vartheta(\U)$. It follows that, except
possibly on an event of probability at most $\delta$, we have $\partial \mathfrak{H}'\cap B_\vartheta(\mathfrak{H}')=(\partial_0\mathfrak{H}'\cup\partial_1\mathfrak{H}')\cap B_\vartheta(\mathfrak{H}')$, for every $\vartheta\in (0,\eta)$.
By scale invariance, the latter equality must hold for every $\eta>0$ (except on an event of probability at most $\delta$). Since $\delta$ was arbitrary, we conclude
that $\partial \mathfrak{H}'=\partial_0\mathfrak{H}'\cup\partial_1\mathfrak{H}'$.

The last assertion of the theorem is easy. 
Let $x\in\mathfrak{H}^\prime$. Then, we have $D_\infty(x,y)\geq \ell^\infty_x$ for every $y\in\partial_1\mathfrak{H}^\prime$ from the bound 
$D_\infty(x,y)\geq| \ell^\infty_x-\ell^\infty_y|$ and the fact that $\ell^\infty_y=0$ if $y\in\partial_1\mathfrak{H}^\prime$. Conversely, let $t\geq 0$
be such that $x=\Pi_\infty(\mathcal{E}^\infty_t)$, and let $r=\inf\{s\geq t: \mathcal{E}^\infty_s\in \partial\mathfrak{T}_\infty\}$. Then, $y:=\Pi_\infty(\mathcal{E}^\infty_r)$
belongs to $\partial_1\mathfrak{H}^\prime$, and it is easy to verify that $D_\infty(x,y)=\ell^\infty_x$. \endproof

In the last part of this section,  we state and prove a consequence of Theorem \ref{new-const-half-plane} that will be 
useful in the next section when we discuss hulls centered on a boundary segment. For simplicity, we write $\xx'=\Lambda'(0)$ for the 
distinguished point of $\HH^\prime$.

 \begin{proposition}
 \label{geo-half}
Let $\eta>0$. Then, almost surely, there exists $\delta>0$ such that the following holds. For every 
$x\in\HH'$ with $D_\infty(\xx',x)\geq \eta$, there is a shortest path from $x$ to $\partial_1\HH'=\{\Lambda'(t):t\in(-\infty,0]\}$
that ends at a point of $\{\Lambda'(t):t\in(-\infty,-\delta]\}$.
\end{proposition}

\proof For $a\in \mathfrak{T}_\infty$, let $s$ be the smallest time such that $\mathcal{E}^\infty_s=a$,
and, for every $u\in[0,\ell^\infty_a]$, set 
$$\phi_a(u)=\sup\{v\leq s: \ell^\infty_{\mathcal{E}^\infty_v}=u\}$$
so that $\mathcal{E}^\infty_{\phi_a(u)}$
 is the ``last'' point before $a$ with label $u$. It easily follows from our definitions
 that $D^\circ_\infty(\mathcal{E}^\infty_{\phi_a(u)},\mathcal{E}^\infty_{\phi_a(v)})=v-u$ for every $0\leq u\leq v\leq \ell^\infty_a$. 
 Note that $\mathcal{E}^\infty_{\phi_a(\ell_a^\infty)}=a$, and $\mathcal{E}^\infty_{\phi_a(0)}\in\partial\mathfrak{T}_\infty$, so that
 $\Pi_\infty(\mathcal{E}^\infty_{\phi_a(0)})\in \partial_1\HH'$. Recalling that $D_\infty(\Pi_\infty(a),\partial_1\HH')=\ell^\infty_a$,
 we obtain that the path $(\Pi_\infty(\mathcal{E}^\infty_{\phi_a(\ell^\infty_a-u)}),0\leq u\leq \ell^\infty_a)$ is a shortest path from $\Pi_\infty(a)$
 to $\partial_1\HH'$. Let us write $\Phi_a$ for this path.
 
 Then, almost surely, we can find $\ve>0$ such that, for any $x\in\HH'$ with $D_\infty(\xx',x)\geq \eta$,
 we have $x=\Pi_\infty(a)$ for some $a\in \mathfrak{T}_\infty\backslash\mathfrak{T}_\infty^\ve$, with the notation
 introduced before Lemma \ref{inf:T:esp}. But then, the geodesic $\Phi_a$ hits $\partial\HH'_1$
 at a point of the form $\Pi_\infty(\mathcal{E}^\infty_r)$ with $r\geq r_\ve$, where $\ee^\infty_{r_\ve}$ 
 is the last point of $\mathfrak{T}_\infty^\ve$ with zero label. By Lemma \ref{inf:T:esp} and the support
 property of exit local times, we have $\Pi_\infty(\ee^\infty_{r_\ve})=\Lambda(-\delta)$
 for some $\delta>0$, and it follows that the path $\Phi_a$ hits $\partial_1\HH'$
 at a point of $\{\Lambda'(t):t\in(-\infty,-\delta]\}$. \endproof

\section{Hulls centered on a boundary segment}
\label{sec:segment}

In this section, we give an analog of Theorem \ref{peeling-HP} for hulls centered on a segment of the 
boundary.  
Our motivation is to investigate various peeling explorations of the Brownian half-plane, and we expect
peeling from a boundary segment to be one of the building blocks for such investigations. Note that the case of peeling from a single 
boundary point is Theorem \ref{peeling-HP}, and that peeling from the whole boundary is discussed in \cite[Section 5.3]{spine}.
Eventually, it would be interesting to characterize all metric explorations of the Brownian half-plane that can be obtained 
by combining different types of peelings (in such a way that at each step the ``unknown'' region remains a 
Brownian half-plane) and possibly taking limits of such explorations. We hope to study these questions in future work.

We consider the curve-decorated Brownian half-plane $(\HH,D,V,\Lambda)$.
Let $r>0$ and $s>0$. Let $\check B^\circ_r(\HH,[0,s])$ be the unique 
unbounded component of the open set
$$\{x\in\HH : D(x,\Lambda([0,s]))>r\}.$$
We write  $\check B^\bullet_r(\HH,[0,s])$ for the closure of  $\check B^\circ_r(\HH,[0,s])$, and we set $ B^\circ_r(\HH,[0,s])=\HH \backslash \check B^\bullet_r(\HH,[0,s])$
and $ B^\bullet_r(\HH,[0,s])=\HH \backslash \check B^\circ_r(\HH,[0,s])$.
We also set
$$\xx_s:=\Lambda(\inf\{t\in \R: \Lambda(t)\notin \check B^\circ_r(\HH,[0,s])\}).$$

\begin{theorem}
\label{markov-seg}
The intrinsic metric on $\check B^\circ_r(\HH,[0,s])$ (associated with the metric $D$ on $\HH$)
has a continuous extension to $\check B^\bullet_r(\HH,[0,s])$, which is a metric on $\check B^\bullet_r(\HH,[0,s])$.
Then the space $\check B^\bullet_r(\HH,[0,s])$ equipped with this extended intrinsic metric, with the
restriction of the volume measure $V$ and with the distinguished point $\xx_s$ is a Brownian half-plane.
\end{theorem}

\noindent{\it Remark.} By analogy with Theorem \ref{indep-peeling}, one can also prove that the
Brownian half-plane $\check B^\bullet_r(\HH,[0,s])$ of the theorem is independent of the 
space $B^\bullet_r(\HH,[0,s])$ equipped with an appropriately defined intrinsic metric, with the restriction of
the volume $V$ and with the distinguished point $\xx_s$. We will however leave this extension
to the reader. 

\smallskip
We state a lemma before proving Theorem \ref{markov-seg}.

\begin{lemma}
\label{markov-seg-lem}
Almost surely, there exists $\delta\in(0,s)$ such that, for every $u\in[s-\delta,s]$,
$$\check B^\bullet_r(\HH,[0,u])=\check B^\bullet_r(\HH,[0,s]).$$
\end{lemma}

\proof We claim that, almost surely, for any compact subset $K$ of $\HH$ not intersecting $\Lambda([0,s])$, there exists $\delta\in(0,s/2)$ such that 
\begin{equation}
\label{seg-tec}
D(x,\Lambda([0,s-\delta]))=D(x,\Lambda([0,s]))
\end{equation}
for every $x\in K$. The statement of the lemma easily follows from \eqref{seg-tec}
by taking $K=\partial B^\bullet_r(\HH,[0,s])$. 

Let us prove \eqref{seg-tec}. We first fix $\eta>0$ so that,
if $y\in\HH$ is such that $D(y,\Lambda([s/2,s]))\leq \eta$, then $D(y,\Lambda((-\infty,0]))>D(y,\Lambda([s/2,s]))$. 
Taking $\eta$ smaller if necessary, we may assume that $D(y,\Lambda([0,s]))\geq \eta$ for every $y\in K$.
Thanks to Proposition \ref{geo-half}, we can find $\delta\in(0,s/2)$
such that, for every $x\in \HH$ with $D(x,\Lambda(s))\geq \eta$, there is a geodesic from
$x$ to $\Lambda((-\infty,s])$ that ends at a point of $\Lambda((-\infty,s-\delta])$ (we use the fact that 
$(\HH,D,V,\Lambda(s+\cdot))$ has the same distribution as $(\HH,D,V,\Lambda)$).
Then, let $x\in K$, and consider a geodesic $\phi$ from $x$ to $\Lambda([0,s])$. If the geodesic $\phi$ ends at a point 
of $\Lambda([0,s-\delta])$, \eqref{seg-tec} clearly holds. Otherwise, on the geodesic $\phi$, we can find 
a point $z$ at distance $\eta$ from $\Lambda([s/2,s])$, and then our choice of $\eta$ implies that a
geodesic from $z$  to $\Lambda([0,s])$ must also be a geodesic from
$z$ to $\Lambda((-\infty,s])$. By our choice of 
$\delta$, the part of the geodesic $\phi$
between $z$ and $\Lambda([0,s])$ can be replaced (without increasing its length) by a geodesic from $z$ to $\Lambda((-\infty,s-\delta])$,
which must end at a point of $\Lambda([0,s-\delta])$, again by the choice of $\eta$.
We conclude that \eqref{seg-tec} holds. \endproof

\proof[Proof of Theorem \ref{markov-seg}.]
From Theorems \ref{peeling-HP} and \ref{indep-peeling}, we know
that 
$$\HH_1=\HH\backslash B^\circ_1(\HH)$$
equipped with the (extended) intrinsic metric, with the restriction of the volume measure on $\HH$
and with the boundary curve $\Lambda^1$
is a curve-decorated Brownian half-plane, which furthermore is independent of the hull $B^\bullet_1(\HH)$ also viewed as
a curve-decorated measure metric space (for the appropriate intrinsic metric). In particular, 
the curve-decorated Brownian half-plane $\HH_1$ is independent of the perimeter $Z_1$ of $B^\bullet_1(\HH)$, which we
know to be exponentially distributed with parameter $3/2$ (Proposition \ref{formulas!}). Note that, by definition,
$\Lambda^1([0,Z_1])=\partial B^\bullet_1(\HH)$.

Now observe that
$$\check B^\circ_r(\HH_1,[0,Z_1])=\HH_1\backslash B^\bullet_r(\HH_1,[0,Z_1])=\HH\backslash B^\bullet_{1+r}(\HH),$$
because a curve connecting a point of $\HH_1$ to infinity stays at distance greater than $r$ from $\Lambda^1([0,Z_1])=\partial B^\bullet_1(\HH)$ 
if and only if it stays at distance greater than $1+r$ from $\xx$. Fom the last display 
and Theorems \ref{peeling-HP} and \ref{indep-peeling} applied to $\mathfrak{H}_{1+r}$, we infer that:
\begin{itemize}
\item[\rm(i)] almost surely, the intrinsic metric on $\check B^\circ_r(\HH_1,[0,Z_1])$
has a continuous extension to its closure $\check B^\bullet_r(\HH_1,[0,Z_1])$, which is a metric on $\check B^\bullet_r(\HH_1,[0,Z_1])$;
\item[\rm(ii)] $\check B^\bullet_r(\HH_1,[0,Z_1])$ equipped with this extended 
metric (and with the restriction of the volume measure on $\HH$, and the distinguished point $\xx_{1+r}$) is a Brownian half-plane, which
is independent of $Z_1$.
\end{itemize}
In particular, the independence property in (ii) holds because the half-plane $\HH\backslash B^\circ_{1+r}(\HH)$ is independent of the hull $B^\bullet_{1+r}(\HH)$
and therefore also of $B^\bullet_1(\HH)$. 

We then observe that (i) still holds if $Z_1$ is replaced by a fixed value $s>0$. Indeed, if this was not true, Lemma \ref{markov-seg-lem}
would allow us to find $\delta\in(0,s)$ such that property (i) written with $Z_1$ replaced by $u$ would fail for every $u\in[s-\delta,s]$ with positive probability. Clearly this
is a contradiction since $\P(s-\delta\leq Z_1\leq s)>0$. This gives the first assertion of the theorem.

To get the second assertion, let $g$ be a bounded continuous function on $\R_+$, and let $F$ be a bounded continuous function on the space $\M^\infty_{\mathrm{GHPU}}$.
Also write $\Theta$ for the distribution of the 
Brownian half-plane (viewed as a pointed measure metric space). It follows from (i) and (ii) that
$$\E[g(Z_1)\,F(\check B^\bullet_r(\HH_1,[0,Z_1]))]= \E[g(Z_1)]\,\Theta(F)= \frac{3}{2}\,\Theta(F)\,\int_0^\infty g(s)\,e^{-3s/2}\,\dd s.$$
On the other hand, we have also, using the independence of $\HH_1$ and $Z_1$,
$$\E[g(Z_1)\,F(\check B^\bullet_r(\HH_1,[0,Z_1]))]=\frac{3}{2}\,\int_0^\infty g(s)\,\E[F(\check B^\bullet_r(\HH_1,[0,s]))]\,e^{-3s/2}\,\dd s.$$
It follows that, for Lebesgue almost every  $s>0$,
$$\E[F(\check B^\bullet_r(\HH_1,[0,s]))]= \Theta(F),$$
or equivalently, since $\HH_1$ and $\HH$
have the same distribution,
$$\E[F(\check B^\bullet_r(\HH,[0,s]))]= \Theta(F).$$
To complete the proof, we need to verify that this property in fact holds for {\it every} $s>0$. To this end, 
we just notice that the mapping $u\mapsto \E[F(\check B^\bullet_r(\HH,[0,u]))]$ is left-continuous, 
as an immediate consequence of Lemma \ref{markov-seg-lem}.  \endproof

\section*{Appendix}

In this appendix, we prove Lemmas \ref{tec-formu} and \ref{tec-formu2}. Recall that $(\mathcal{R}_t)_{t\geq 0}$ is a five-dimensional Bessel process started from $0$.

\proof[Proof of Lemma \ref{tec-formu}] Fix $0<x<y$ and $c$ such that $0<c<x$, and write $\wt\r_t=\r_{(\mathcal{L}_y-t)\vee 0}$ for every $t\geq 0$. Also set
$T_x=\mathcal{L}_y-\mathcal{L}_x=\inf\{t\geq 0: \wt \r_t=x\}$. As a consequence of Nagasawa's time-reversal theorem, we know that $(\wt\r_t)_{t\geq 0}$ is a Bessel process of dimension $-1$
started at $y$, and  we have
$$\E\Big[\exp\Big(-\int_{\mathcal{L}_x}^{\mathcal{L}_y} \dd t\, \Big((\mathcal{R}_t-c)^{-2}-(\mathcal{R}_t)^{-2}\Big)\Big)\Big] 
=\E\Big[\exp\Big(-\int_0^{T_x}\,\dd t \Big((\wt\r_t-c)^{-2}-(\wt\r_t)^{-2}\Big)\Big)\Big].$$
For every $0<u<v$, let $B=(B_t)_{t\geq 0}$ stand for a linear Brownian motion thats starts at $v$ under the probability measure $\P_v$
and let $T^{(B)}_u=\inf\{t\geq 0:B_t=u\}$. Using the absolute continuity properties of Bessel processes with respect
to Brownian motion (see e.g. \cite[Lemma 1]{Bessel}), we have
\begin{align*}
\E\Big[\exp\Big(-\int_0^{T_x}\,\dd t \Big((\wt\r_t-c)^{-2}-(\wt\r_t)^{-2}\Big)\Big)\Big]&=\frac{y}{x}\,\E_y\Big[\exp\Big(-\int_0^{T^{(B)}_x} \dd t\,(B_t-c)^{-2}\Big)\Big]\\
&=\frac{y}{x}\,\E_{y-c}\Big[\exp\Big(-\int_0^{T^{(B)}_{x-c}} \dd t\,(B_t)^{-2}\Big)\Big]\\
&=\frac{y}{x}\times \frac{x-c}{y-c},
\end{align*}
where the last equality is classical and follows from an application of the optional stopping theorem to the martingale
$$\frac{1}{B_{t\wedge T^{(B)}_{x-c}}}\,\exp\Big(-\int_0^{t\wedge T^{(B)}_{x-c}} \dd s\,(B_s)^{-2}\Big)$$
under $\P_{y-c}$. 
\endproof

\proof[Proof of Lemma \ref{tec-formu2}] We use the same notation as in the proof of Lemma \ref{tec-formu}, taking now $y=r/\sqrt{3}$. In particular,
$\wt\r_t=R_{(\mathcal{L}_y-t)\vee 0}$, and $T_x=\inf\{t\geq 0: \wt\r_t=x\}$ for every $x\in[0,y]$. By the same time-reversal argument, we have 
\begin{align*}
&E\Big[\exp\Big(-\nu \gamma_r - 2 \int_0^{\gamma_r} \dd t\,g_\mu(\sqrt{3}\mathcal{R}_t)\Big)\Big] \\
&\quad=\E\Big[\exp\Big(-\int_0^{T_0} \dd t\,(\nu+2g_\mu(\sqrt{3}\wt\r_t))\Big)\Big]\\
&\quad=\lim_{\ve \to 0} \E\Big[\exp\Big(-\int_0^{T_\ve} \dd t\,(\nu+2g_\mu(\sqrt{3}\wt\r_t))\Big)\Big]\\
&\quad=\lim_{\ve\to 0} \frac{y}{\ve} \,\E_y\Big[\exp\Big(-\int_{0}^{T^{(B)}_\ve} \dd t\,(\nu+2g_\mu(\sqrt{3}B_t)+(B_t)^{-2})\Big)\Big],
\end{align*}
where the last equality relies on the same absolute continuity argument as in the proof of Lemma \ref{tec-formu}. 

Observe that 
$$\nu+2g_\mu(\sqrt{3}B_t)+(B_t)^{-2}=
\nu+\sqrt{2\mu}\Big(3\,\coth\big((2\mu)^{1/4}\sqrt{3} B_t\big)^2 - 2\Big) .$$
Set $a:=(2\mu)^{1/4}$, and then, for every $s>0$,
$$f(s):=\nu + a^2\Big( 3\coth(a\sqrt{3}s)^2 -2\Big).$$
Then, if 
$$F(s):= \exp\Big(-s\sqrt{2(a^2+\nu)}\Big)\,\Big(a\coth(a\sqrt{3}s)+ \sqrt{\frac{2}{3}(a^2+\nu)}\Big),$$
a direct calculation shows that
$$F''(s)=2\,f(s)\,F(s).$$
By a simple application of It\^o's formula, it follows that, for $\ve\in(0,y)$,
$$F(B_{t\wedge T^{(B)}_\ve})\,\exp\Big(-\int_0^{t\wedge T^{(B)}_\ve} \dd s\, f(B_s)\Big)$$
is a (bounded) martingale under $\P_y$. The optional stopping theorem then gives
$$\E_y\Big[\exp\Big(-\int_0^{T^{(B)}_\ve} \dd t\, f(B_t)\Big)\Big] = \frac{F(y)}{F(\ve)},$$
and thus
\begin{align*}
\frac{y}{\ve} \,\E_y\Big[\exp\Big(-\int_{0}^{T^{(B)}_\ve} \dd t\,(\nu+2g_\mu(\sqrt{3}B_t)+(B_t)^{-2})\Big)\Big]
&=\frac{y}{\ve} \,\E_y\Big[\exp\Big(-\int_0^{T^{(B)}_\ve} \dd t\, f(B_t)\Big)\Big]\\
&=\frac{y\,F(y)}{\ve\,F(\ve)},
\end{align*}
which converges when $\ve \to 0$ to
$\sqrt{3}\,y\,F(y)= r\,F(r/\sqrt{3})=G_r(\mu,\nu)$. This completes the proof. 
\endproof

{\bf Acknowledgements.} We are indebted to Thomas Budzinski for stimulating conversations.
We also thank an anonymous referee for several useful remarks.

\end{document}